\documentclass{amsart}

 \usepackage[fontsize=11pt]{fontsize}  

    \usepackage[a4paper, left=3cm, right=3cm, top=2cm,bottom=2cm]{geometry}  

\numberwithin{equation}{section}
\usepackage{amsmath}
\usepackage{amsthm}
\usepackage{tikz}
\usepackage[all]{xy}
\usepackage{amssymb}
\usepackage{mathtools}
\usepackage{algorithm}
\usepackage{algpseudocode}
\usepackage{stmaryrd}
\usepackage[toc]{appendix}
\usepackage{url}
\usepackage{enumitem}

  \usepackage[normalem]{ulem} 

\usepackage{color}
\usepackage[linktocpage=true, hidelinks, colorlinks, linkcolor=blue, citecolor=magenta, bookmarksopen=true, urlcolor=brown]{hyperref}
\usepackage[nameinlink]
{cleveref}

\DeclareMathAlphabet{\mathbbold}{U}{bbold}{m}{n}
\usepackage[numbers,sort&compress]{natbib}
\usepackage{tikz-cd}
\usepackage[all]{xy}
\newtheorem{theorem}{Theorem}[section]
\newtheorem{lemma}[theorem]{Lemma}

\newtheorem{corollary}[theorem]{Corollary}
\newtheorem{proposition}[theorem]{Proposition}

\theoremstyle{definition}
\newtheorem{definition}[theorem]{Definition}
\newtheorem{question}[theorem]{Question}
\newtheorem{example}[theorem]{Example}

\newtheorem{remark}[theorem]{Remark}
\newtheorem{notation}[theorem]{Notation}
  \newtheorem{construction}[theorem]{Construction}

\usepackage{graphicx} 
\newcommand{\mc}{\mathcal}
\newcommand{\mt}{\mathrm}

\newcommand{\wh}{\widehat}
\newcommand{\wt}{\widetilde}

\newcommand{\ZZ}{\mathbf{Z}}
\newcommand{\QQ}{\mathbf{Q}}

\newcommand{\NN}{\mathbf{N}}
\newcommand{\Fil}{\mt{Fil}}
\newcommand{\gr}{\mt{gr}}
\newcommand{\ad}{\mt{ad}}
\newcommand{\Hom}{\mt{Hom}}

\newcommand{\Lie}{\mt{Lie}}
\newcommand{\coker}{\mt{coker}}
\newcommand{\GL}{\mt{GL}}
\newcommand{\Sym}{\mt{Sym}}
\newcommand{\End}{\mt{End}}
\newcommand{\Gal}{\mt{Gal}}

\newcommand{\Mat}{\mt{Mat}}

\newcommand{\Id}{\mt{Id}}

\newcommand{\dR}{\mt{dR}}
\newcommand{\val}{\mt{val}}

\newcommand{\an}{\mt{an}}
\newcommand{\cyc}{\mt{cyc}}
\newcommand{\R}{\mt{R}}
\newcommand{\la}{\mt{la}}

\newcommand{\C}{\mathbf{C}}

\newcommand{\B}{\mathbf{B}}
\newcommand{\han}{\text{$\mbox{-}\mathrm{an}$}}
\newcommand{\hfin}{\text{$\mbox{-}\mathrm{fin}$}}

\newcommand{\hla}{\text{$\mbox{-}\mathrm{la}$}}

\newcommand{\pa}{\mt{pa}}

\newcommand{\Sen}{\mt{Sen}}
\newcommand{\D}{{\bf{D}}}
\newcommand{\BdRplus}{\B_{\dR}^{+}}
\newcommand{\DdRplus}{\D_{\dR}^{+}}
\newcommand{\LieGamma}{\Lie\Gamma}
\newcommand{\BdRpluspa}{\B_{\dR}^{+}(K_\infty)^{\pa}}
\newcommand{\Kinftyla}{\wh{K}_{\infty}^\la}
\newcommand{\Linftyla}{\wh{L}_{\infty}^\la}
\newcommand{\Ktaginftyla}{\wh{K}_{\infty}'^\la}
\newcommand{\kft}{{K_{\mathrm{fT}}}}
\newcommand{\kftla}{{\wh{K}^\la_{\mathrm{fT}}}}

\usepackage{xcolor}

\newcommand \onto {\twoheadrightarrow}

\newcommand{\rep}{{\mathrm{Rep}}}

\newcommand{\gal}{{\mathrm{Gal}}}

\newcommand{\kinfty}{{K_{\infty}}}

\newcommand{\hatkinfty}{{\widehat{K}_{\infty}}}

\newcommand{\gammak}{{\Gamma(\kinfty)}}

\newcommand{\gk}{{G_K}}

\newcommand{\dla}{\text{$\mbox{-}\mathrm{la}$}}

\newcommand{\dpa}{\text{$\mbox{-}\mathrm{pa}$}}

\newcommand{\bdrplus}{{\mathbf{B}^+_{\mathrm{dR}}}}

\newcommand{\Qp}{{\mathbf{Q}_p}}

\newcommand{\qp}{{\mathbf{Q}_p}}

\newcommand{\barK}{{\overline{K}}}

\newcommand{\hatkinftyla}{{\widehat{K}_{\infty}^\la}}

\newcommand{\kcyc}{{K_\cyc}}

\newcommand{\hk}{{H(K_\infty)}}

\newcommand{\rg}{{\R\Gamma}}

\newcommand{\bdrplusm}{{\mathbf{B}^{+}_{\dR, m}}}
\newcommand{\bdrplusmkinftyla}{{\mathbf{B}^{+}_{\dR, m}(\kinfty)^\la}}

\newcommand{\bdrpluskinftypa}{{\mathbf{B}^{+}_{\dR}(\kinfty)^\pa}}

 \usepackage{relsize} 
\usepackage{lmodern} 
\usepackage{setspace} 

\title[]{de Rham theory and locally analytic vectors}

\author[]{Hui Gao}
\address{Department of Mathematics and Shenzhen International Center for Mathematics, Southern University of Science and Technology, Shenzhen 518055, China}   \email{gaoh@sustech.edu.cn}

\author[]{Gal Porat}
\address{Department of Mathematics, Weizmann Institute of Science, Rehovot, Israel}   \email{galporat1@gmail.com}

\author[]{Léo Poyeton}
\address{Université de Bordeaux, Institut de Mathématiques de Bordeaux, France}   \email{leo.poyeton@math.u-bordeaux.fr}

\begin{document}

\begin{abstract}
 \normalsize{  
 Let \(K_\infty/K\)  be a \(p\)-adic Lie extension of a $p$-adic field $K$. We study the subring of pro-analytic vectors in the  de Rham period ring \(\B_{\mathrm{dR}}^+(K_\infty)\). 
We show that the pro-analytic subring admits a Galois-equivariant isomorphism with a  formal power series ring $\Kinftyla \llbracket t_{K_\infty} \rrbracket$ if and only if $K_\infty$ satisfies a certain orientability condition, which   says that the  \(\widehat{K}_\infty\)-level Sen operator  admits  a Galois-equivariant $\B_{\mathrm{dR}}^+$-lift. 
A key input  is the vanishing of   higher locally analytic vectors of    \(\widehat{K}_\infty\)-representations.  
As an application, we show that the lifted Sen operator induces regular connections on pro-analytic vectors of $\B_{\mathrm{dR}}^+$-representations, and can be used to compute Galois cohomology.
 }
\end{abstract}

\subjclass[2020]{Primary 11F80; 14F30.}
 \keywords{}
 
 \date{\today} 
 \maketitle
\setcounter{tocdepth}{2}
\begingroup
\small
\tableofcontents 
\endgroup
\section{Introduction}

This article is concerned with the locally analytic structure of de Rham
period rings attached to $p$-adic Lie extensions.
Let $K$ be a mixed characteristic complete discrete valuation field with perfect residue
field of characteristic $p$ and let $\gk=\gal(\barK/K)$ be the absolute Galois group. Let $K_\infty/K$ be an infinitely ramified
Galois extension such that
$\Gamma(K_\infty):=\operatorname{Gal}(K_\infty/K)$
is a $p$-adic Lie group. Write
$H(K_\infty)=\operatorname{Gal}(\overline K/K_\infty)$. Our aim is to understand the pro-analytic vectors of the action of $\Gamma(K_\infty)$ on $\B_{\mathrm{dR}}^+(K_\infty):=(\B_{\mathrm{dR}}^+)^{H(K_\infty)}$. A key observation is that this problem is controlled by the de Rham theory of the Lie algebra $\operatorname{Lie}\Gamma(K_\infty)$ with the adjoint action of
$\Gamma(K_\infty)$, viewed as a $G_K$-representation by precomposition with
$G_K\twoheadrightarrow\Gamma(K_\infty)$. Before stating the main result, we
explain how this problem arises, why the familiar cyclotomic answer does not
extend formally to an arbitrary tower, and how the adjoint representation
$\operatorname{Lie}\Gamma(K_\infty)$ enters the picture.

\subsection{$p$-adic Lie towers and analytic period rings}

The choice of an infinitely ramified extension of $K$ is a recurring feature of
$p$-adic Hodge theory, Iwasawa theory, and the $p$-adic Langlands program.
Different choices of $K_\infty$ lead to rather different theories.  We
return to the principal examples after stating the main theorem.

The completion $\widehat K_\infty$ is the most accessible period ring
attached to the tower, and in some sense the computation 
$\widehat K_\infty^{\mathrm{la}}=K_\infty$ in the cyclotomic setting via Tate's trace maps (\cite{tate1967p}) can be considered as a starting point of $p$-adic Hodge theory. One of the main results of Berger and Colmez's article \cite{berger2016theorie} is the computation of $\widehat K_\infty^{\mathrm{la}}$ for general $K_\infty$. 

However, one would like to understand locally
analytic vectors in much larger period rings.  Perhaps the most interesting such period ring is the large
Robba ring
$(\widetilde{\mathbf B}_{\mathrm{rig}}^\dagger)^{H(K_\infty)}$.
In the cyclotomic case, Berger (\cite{berger2016multivariable}) proves that its pro-analytic vectors recover, up to iterates of
Frobenius, the usual Robba ring.  Further calculations for other towers and other rings
of periods were carried out by L. P.  (\cite{poyeton_locally_analytic_vectors_and_rings_of_periods},\cite{poyeton_locally_analytic_vectors_and_zp_extensions}).
These Robba rings are important because they form a bridge from Galois
representations to $p$-adic analysis.  In the cyclotomic theory they lead to connections with
$p$-adic differential equations (\cite{berger2002representations},\cite{cherbonnier1998representations},\cite{fontaine2007representations},\cite{kedlaya2005slope}), to explicit reciprocity laws
and Perrin--Riou's big exponential in Iwasawa theory (\cite{berger2003bloch}), and to
the description of locally analytic representations in the $p$-adic local
Langlands correspondence (\cite{colmez2016representations},\cite{colmez2014p}).  Thus, from the point of view of
applications, the locally analytic vectors in Robba-type period rings are
the objects one would ultimately like to compute.

For general towers $K_\infty/K$, such a computation appears to be much more
difficult. Results of L.~P.~  (\cite{poyeton_locally_analytic_vectors_and_rings_of_periods}) and Steingart (\cite{steingart2024higher}) suggest $(\widetilde{\mathbf B}_{\mathrm{rig}}^\dagger)^{H(K_\infty)}$ has nonzero higher locally analytic vectors whenever $K_\infty$ does not contain the cyclotomic extension. This indicates that the correct locally analytic Robba ring for applications for general $K_\infty/K$ is likely to be the derived version $(\widetilde{\mathbf B}_{\mathrm{rig}}^\dagger)^{H(K_\infty),\mt{R}\Gamma(K_\infty)\hla}$. Our understanding is that strong theorems in this direction have been proven by Anschütz, Le Bras, Rodríguez Camargo and  Scholze in their upcoming work on analytic prismatization, and they will be applied to investigate Iwasawa theory for general $K_\infty$ in upcoming work of Ponsinet, Rodrigues Jacinto and Steingart.

The ring $\mathbf B_{\mathrm{dR}}^+(K_\infty)$ has an intermediate
level of complexity between $\wh{K}_\infty$ and $(\widetilde{\mathbf B}_{\mathrm{rig}}^\dagger)^{H(K_\infty)}$.  It is larger than $\widehat K_\infty$ and remembers the entire de
Rham filtration, but its reduction modulo its maximal ideal  $I_\theta$ is
$\widehat K_\infty$, and its successive graded pieces are rank-one
semilinear $\widehat K_\infty$-representations.  It is therefore natural to
ask whether the calculation of Berger--Colmez can be lifted through the
$I_\theta$-adic filtration.  This is the problem studied in the present
article.

\subsection{From Sen theory to de Rham theory} \label{subsec:intro_sen_dR} 
 
 One can study $\wh{K}_\infty$-semilinear representations of $\Gamma(K_\infty)$ using Sen's theory. The original
theory uses vectors which lie in finite-dimensional $K$-subspaces stable
under the Galois action.  This works especially well when $\Gamma(K_\infty)$ is one-dimensional, but finite vectors become too small in higher
dimension.  Berger and Colmez (\cite{berger2016theorie}) observed that locally analytic vectors are
the correct replacement. 
 The most difficult part of their article is the computation
of $\widehat K_\infty^{\mathrm{la}}$ itself.  They show that $\widehat K_\infty^{\mathrm{la}}$ is a field of convergent power series in $\dim\Gamma(K_\infty)-1$ variables over $K_\infty$.

It is natural to ask for a de Rham version of this theorem.  Reduction
modulo the maximal ideal $I_\theta$ identifies $ \mathbf
B_{\mathrm{dR}}^+(K_\infty)/I_\theta
       \cong\widehat K_\infty$.
The quotients
$\mathbf B_{\mathrm{dR}}^+(K_\infty)/I_\theta^n$ are Banach
representations of $\Gamma(K_\infty)$, and
$\mathbf B_{\mathrm{dR}}^+(K_\infty)$ is their inverse limit.  The natural object to consider in this setting is then the pro-analytic vectors $\mathbf B_{\mathrm{dR}}^+(K_\infty)^{\pa} = \varprojlim_n [\mathbf B_{\mathrm{dR}}^+(K_\infty)/I_\theta^n]^\la$.

In the cyclotomic case (i.e. $\kinfty=\kcyc$), G.~P.~ (\cite[Proposition 2.6]{porat2022lubin}) computed $\mathbf B_{\mathrm{dR}}^+(\kcyc)^{\pa}=\kcyc\llbracket t\rrbracket,$ where $t$ is Fontaine's element, where a key point is that there is a natural inclusion $\kcyc \subset \barK \subset \mathbf B_{\mathrm{dR}}^+$.
In general, however, $\Kinftyla/K$ may not be an algebraic extension, so the structure of $\mathbf B_{\mathrm{dR}}^+(\kinfty)^{\pa}$ becomes more complicated. An argument using Cohen's structure theorem shows that $\BdRpluspa$ is non-canonically isomorphic to a power series ring $\Kinftyla\llbracket t_{K_\infty}\rrbracket $, where $t_{K_\infty}$ is a variable, see Theorem \ref{thm:dvr}.  This can be regarded as an analogue of the non-canonical ring isomorphism $\bdrplus \cong \C\llbracket t \rrbracket$, and thus prompts the following natural canonicity question, which will be the core of our investigation. 

\begin{question}
\label{question:bdr_plus_iso}
When is there a $\Gamma(K_\infty)$-equivariant\footnote{For simplicity, we are being imprecise in the introduction: one needs to explain how to equip $\Kinftyla \llbracket t_{K_\infty} \rrbracket$ with a natural $\Gamma(K_\infty)$-action. It is in fact slightly more natural to consider instead the completed symmetric algebra  $\wh{\Sym}_{\Kinftyla}(\gr^1\BdRplus(K_\infty))^\la)$, see~ Theorem \ref{thm:brdplusps_iso}; the $\Gamma(K_\infty)$-action on $\Kinftyla\llbracket t_{K_\infty}\rrbracket$ then depends on a trivialization $\gr^1\BdRplus(K_\infty))^\la \cong \Kinftyla \cdot t_{K_\infty}$.} isomorphism
of the form 
\begin{equation*} \label{eqintro:bdr_plus_iso}
\BdRpluspa \cong \Kinftyla\llbracket t_{K_\infty}\rrbracket ? 
\end{equation*} 
In particular, when is there   a $\Gamma(K_\infty)$-equivariant section 
\[ \Kinftyla \to \BdRpluspa ?\]
\end{question}


It turns out this question has a deformation-theoretic flavour, and its answer lies in the de Rham theory of the adjoint representation. Indeed, consider
$\operatorname{Lie}\Gamma(K_\infty)$ with the adjoint action of
$\Gamma(K_\infty)$, viewed as a $G_K$-representation by precomposition with
$G_K\twoheadrightarrow\Gamma(K_\infty)$.  A theorem of Sen (\cite[Theorem 12]{sen1980continuous}) gives a
canonical nonzero element, the Sen operator
\[
 \Theta(K_\infty)\in
 \bigl(\C\otimes_{\mathbf Q_p}
       \operatorname{Lie}\Gamma(K_\infty)\bigr)^{G_K}=\D_{\C}(\Lie\Gamma(K_\infty)).
\] Berger and Colmez prove that,
when this element acts as a differential operator on
$\widehat K_\infty^{\mathrm{la}}$, it acts by zero (\cite[\S6]{berger2016theorie}). Our idea is to consider it as a $\C$-period of $\Lie\Gamma(K_\infty)$ and try to lift it to $\mathbf B_{\mathrm{dR}}^+$. Thus we are led to the following key definition.

\begin{definition}[Definition \ref{def:orientation}] \label{def:intro_orient}
An \emph{orientation} of $K_\infty/K$ is a lift of
$\Theta(K_\infty)$ to an element
\[
 \nabla(K_\infty)\in
 \bigl(\mathbf B_{\mathrm{dR}}^+\otimes_{\mathbf Q_p}
       \operatorname{Lie}\Gamma(K_\infty)\bigr)^{G_K}=\D_{\dR}^+(\Lie\Gamma(K_\infty)).
\]
We say that $K_\infty/K$ is \emph{orientable} if an orientation exists. We say $K_\infty/K$ is \emph{canonically orientable} if a unique orientation exists.
\end{definition}

It is not difficult to show that orientability is a necessary condition for the existence of a $\Gamma(K_\infty)$-equivariant isomorphism $\BdRpluspa \cong \Kinftyla\llbracket t_{K_\infty}\rrbracket$ in Question \ref{question:bdr_plus_iso} to hold. Indeed, such an isomorphism leads to a $\Gamma(K_\infty)$-equivariant section   map $ \Kinftyla \to \BdRpluspa$, along which one can Galois-equivariantly lift $\Theta(\kinfty) \in  \Kinftyla \otimes_{\QQ_p} \Lie \gammak$ to $\nabla(K_\infty)$. In fact, we shall quickly see that  orientability is also a sufficient condition in Theorem \ref{thm:brdplusps_iso_intro} below.

\begin{example} \label{ex:intro_ori}

\begin{enumerate}
\item  \label{item:ori1}
(Theorem \ref{thm:characterization_of_canonical_orientability}). 
$K_\infty$ is canonically orientable if and only if the representation
 $\Lie\Gamma(K_\infty)$ has no generalized Hodge-Tate weights in $\ZZ_{\leq-1}$.  This includes familiar examples:
 \begin{enumerate}
    \item (Examples \ref{ex:ab_ext} and \ref{ex:nilp_ext}).  Abelian extensions (and nilpotent extensions in general) $K_\infty/K$ are canonically orientable. In particular, cyclotomic, Lubin-Tate and 1-dimensional towers $K_\infty/K$ are canonically orientable. This observation matches with previous study of    $\bdrplus$-representations over these towers (see \cite{Fon04},  \cite{berger2016multivariable}, \cite{porat2022lubin} etc.)

\item  \label{item:ori1b} (Example \ref{ex:Kummer_ext}).  The false Tate (curve)   extension $K(\zeta_{p^\infty},\pi^{1/p^\infty})/K$ is canonically orientable (where $\zeta_{p^\infty}$ resp.~ $\pi^{1/p^\infty}$ signifies taking a compatible system of $p^n$-th roots of 1 and a fixed uniformizer $\pi$). 
This is the tower used in the study of $(\varphi, \tau)$-modules (\cite{caruso2013representations}, \cite{GL20},  \cite{gao_poyeton_overconvergence} etc.), see also the $\bdrplus$-theory developed in \cite{GMWdR}.
 \end{enumerate}

\item \label{item_dR}
(Proposition \ref{prop:de_rham_extensions_are_orientable} and Theorem \ref{thm:characterization_of_orientability}).  If $\Lie \gammak$ is a de Rham representation, then $K_\infty/K$ is orientable (although not necessarily canonically so). More generally, the tower $\kinfty$ is orientable if and only if   $\dim_K\D_{\dR}(\Lie\Gamma(K_\infty)) = \dim_K\D_{\mt{HT}}(\Lie\Gamma(K_\infty)).$ 

\item \label{item_nondR}
(Example \ref{ex_non_dR_ext}).  The well-known non-split extension $0\rightarrow\QQ_{p}\rightarrow V\rightarrow\QQ_{p}(1)\rightarrow0$ induces a $p$-adic Lie extension which is non-orientable (in particular, the representation $\Lie \gammak$ is Hodge--Tate but not de Rham). 

\item (Example \ref{ex:SL2}).  Suppose
$\Gamma(K_\infty)\cong\mathrm{SL}_2(\mathbf Z_p)$.  Berger and Colmez
studied these extensions in \cite[\S5]{berger2016theorie} and described $\Kinftyla$.  The standard two-dimensional
representation $V_{\mt{std}}$ may be considered as a $G_K$-representation by precomposing it with $G_K \twoheadrightarrow \Gamma(K_\infty)$. It has generalized
Hodge--Tate weights $\pm s$ for some
$s\in\overline{K}^\times$. 
Depending  on the value of $s$ and properties of $V_{\mt{std}}$, it turns out $K_\infty/K$ could be either  canonically  orientable, non-canonically orientable or non-orientable.
\end{enumerate}
\end{example}

Our main theorem shows that for every orientation we get a $\Gamma(K_\infty)$-equivariant isomorphism of $\BdRpluspa$ with a 1-variable power series ring over $\Kinftyla$. 

\begin{theorem}[Theorem  \ref{thm:brdplusps_iso}]
\label{thm:brdplusps_iso_intro}
Let $K_\infty/K$ be orientable with a chosen orientation $\nabla=\nabla(K_\infty)$. Then the reduction map $\bdrplus(\kinfty) \to \hatkinfty$ induces a $\Gamma(K_\infty)$-equivariant isomorphism
\[ \BdRplus(K_\infty)^{\pa,\nabla =0}
\cong \Kinftyla.\]
Furthermore, there   exists an  element $t_{K_\infty}\in (\mathrm{Fil}^1 \bdrplus(\kinfty))^{\mathrm{pa}} $ such that $\nabla(t_\kinfty)=t_\kinfty$, and for any such element, the natural inclusion map 
\[ \BdRplus(K_\infty)^{\pa,\nabla =0}\llbracket t_\kinfty\rrbracket \to \BdRpluspa\]
is an isomorphism. 
\end{theorem}

The following theorem then completely answers Question \ref{question:bdr_plus_iso}: an orientation is exactly the data which determines a $\Gamma(K_\infty)$-equivariant power series description of $\BdRpluspa$.  

\begin{theorem}[Theorem \ref{thm:parameterization_of_bdrplus_isos}]
\label{thm:intro_para}
The orientations of $K_\infty$ are naturally in bijection with the set of $\Gamma(K_\infty)$-equivariant isomorphisms between $\BdRpluspa$ and $1$-variable power series rings over $\Kinftyla$ inducing the identity on graded pieces.
\end{theorem}

\begin{remark}\label{rem:intro_can_ori}
In particular, if a tower $K_\infty/K$ is canonically orientable, then there is a \emph{canonical} isomorphism $$\BdRpluspa \cong \Kinftyla\llbracket t_{K_\infty}\rrbracket .$$
\end{remark}


\subsection{A cohomological application}  \label{sec_intro:coho_app}

The key fact $\nabla(t_\kinfty)=t_\kinfty$ in Theorem \ref{thm:brdplusps_iso_intro} makes it possible to connect $\bdrplus$-representations with \emph{regular connections} by using any orientable tower. Let $U$ be a finite free $\bdrplus$-module equipped with a semilinear $\gk$-action; using a standard d\'evissage argument which makes use of our vanishing theorem for higher locally analytic vectors (Theorem \ref{thm:higher_la_vanishing_k_infty}), one can show that
\[ D:= (U^{\hk})^{\gammak\dpa}\]
is a finite free  $\BdRplus(K_\infty)^{\pa}$-module with full rank, which then admits a $\nabla=\nabla(\kinfty)$-action. Using the identification $\BdRplus(K_\infty)^{\pa} \cong \Kinftyla\llbracket t_{K_\infty}\rrbracket$, and using the fact $\nabla(t_\kinfty)=t_\kinfty$, it is clear that $\nabla$ induces a \emph{regular connection} on $D$ in the sense that it satisfies the Leibniz rule with respect to the (regular) derivation $t_\kinfty \cdot \frac{d}{dt_\kinfty}$, that is:
\[ \nabla(ax)=a\nabla(x) + t_\kinfty \cdot \frac{d}{dt_\kinfty}(a)x\]
for all $a \in \Kinftyla\llbracket t_{K_\infty}\rrbracket$ and $x\in D$.
We have the following cohomological consequence.

\begin{theorem}[Theorem \ref{thm:dR_cat_coho}]
Suppose $\kinfty$ is orientable with a chosen orientation $\nabla$. 
Then there is a  $\gammak$-equivariant quasi-isomorphism
\[ [D \xrightarrow{\nabla} D] \simeq \rg(\gk, U)\otimes_K \Kinftyla.\]
In addition, taking $\gammak$-invariants, we have
\[ [D \xrightarrow{\nabla} D]^{\gammak=1} \simeq \rg(\gk, U). \]
\end{theorem}

\subsection{Structure of the article}

In \S2 we prove the vanishing of the higher locally analytic vectors of a
finite free semilinear $\widehat K_\infty$-representation.  This gives the
exactness needed to lift locally analytic vectors through the quotients of
$\mathbf B_{\mathrm{dR}}^+(K_\infty)$.  In \S3 we recall the Sen operator
of a $p$-adic Lie tower and introduce the adjoint representation
$\operatorname{Lie}\Gamma(K_\infty)$.  In \S4 we define orientations,
give criteria for their existence and uniqueness, and study their behavior
under morphisms of towers.  In \S5 we compute the pro-analytic vectors in
$\mathbf B_{\mathrm{dR}}^+(K_\infty)$. In \S6 we deduce consequences of the main theorem for orientations. In \S7 we give cohomological applications in the theory of Sen modules and $\BdRplus$-representations.

\subsection{Notation and conventions}
\label{subsection:notation_and_conventions}
\subsubsection{{$p$}-adic Hodge theory}
\label{subsubsection:p_adic_hodge_theory_notations}

Let $K$ be a mixed characteristic complete discrete valuation field with perfect residue
field of characteristic $p$.  
Let $\kcyc/K$ be the cyclotomic extension by adjoining all $p$-power roots of unity, and let $\chi_{\cyc}$ be the cyclotomic character inducing   $\gal(\kcyc/K) \to \mathbb{Z}_p^\times$.  
Our Hodge-Tate weights are normalized so that $\chi_{\cyc}$ is of weight 1. Let $\C$ be the completion of the algebraic closure   $\overline{K}$. Let $\BdRplus$ be Fontaine's usual DVR with uniformizer $t$. Fontaine's $G_K$-equivariant map $\theta:\BdRplus\to \C$ gives an identification $\BdRplus/t\BdRplus \cong \C$.

Whenever we write $K_{\infty}/K$, we always mean  an infinitely ramified extension such that $\Gamma(K_\infty/K) := \Gal(K_\infty/K)$ is a $p$-adic Lie group. The completion $\wh{K}_\infty$ is a perfectoid field (see \cite[Theorem 2.13]{Coates_Greenberg} and \cite[Remark 3.3]{Sch12}). 
Most of the properties discussed in this article only depend on $\Gamma(K_\infty/K)$ up to an open subgroup, and so only depend on $K$ up to a finite extension; for this reason we sometimes omit $K$ from the notation and simply write $\Gamma(K_\infty)$ for $\Gamma(K_\infty/K)$. We let $\BdRplus(K_\infty):=\B_{\dR}^{+,\Gal(\overline{K}/K_\infty)}$. Then $\BdRplus(K_\infty)$ is a DVR (see~ \cite[Corollary 6.4]{Sch13}) whose maximal ideal we denote by $I_\theta$, and by pro-étale descent one shows that $\theta$ induces an isomorphism
$\BdRplus(K_\infty)/I_\theta\cong\wh{K}_\infty$. Each quotient $\BdRplus(K_\infty)/I_\theta^n$ is a $\QQ_p$-Banach space with a continuous action of $\Gamma(K_\infty)$. The expression $\BdRplus(K_\infty) = \varprojlim_n(\BdRplus(K_\infty)/I_\theta^n)$ exhibits $\BdRplus(K_\infty)$ as a $\QQ_p$-Fréchet space
with a continuous action of $\Gamma(K_\infty)$. We let $L(K_\infty) = \gr^1\BdRplus(K_\infty)$. It is a 1-dimensional $\wh{K}_\infty$-vector space endowed with a $\wh{K}_\infty$-semilinear action of $\Gamma(K_\infty).$

\subsubsection{Locally analytic vectors}
\label{subsubsection:la_vectors_notations}
Given a compact $p$-adic Lie group $G$ acting continuously on a $\QQ_p$-Banach space $M$,
we let $M^{\la}$ denote the $G$-locally analytic vectors in $M$ (the group $G$ is omitted from the notation, since it will always be clear from the context). If $M$ is a $\QQ_p$-Fréchet space with a $G$-action, so that $M = \varprojlim_i M_i$ with the $M_i$ forming a compatible system of $\QQ_p$-Banach representations of $G$, we define the pro-analytic vectors in $M$ as $M^{\pa}:=\varprojlim_i M_i^{\la}$.

The following lemma summarizes two well-known results which are repeatedly used in this paper.   

\begin{lemma}\label{lem:basis_lav}
Let $R$ be a $\qp$-Fr\'echet algebra with $G$-action.
\begin{enumerate} 
\item $(R^\pa)^\times=R^\pa \cap R^\times$. In particular, if $R$ is a field, then so is $R^\pa$. 
\item   Let $M$ be a finite free $R$-module equipped with an $R$-semilinear $G$-action. Suppose $e_1, \cdots, e_d$ is a basis of $M$ such that each $e_i$ lies in $M^{\pa}$.  Then $M^\pa$ is a finite free $R^\pa$-module with $e_1, \cdots, e_d$ as a basis.
\end{enumerate}
\end{lemma}
\begin{proof}
    For item (1), see \cite[Lemma 2.5]{berger2016theorie} and \cite[Lemma 3.1.2]{gao_poyeton_overconvergence}. For item (2), see \cite[Proposition 2.4]{berger2016multivariable}.
\end{proof}
\color{black}

Given a compact $p$-adic Lie group $G$, we denote by $\mt{R}^i_{G \hla}$ for $i \ge 0$ the right derived functors of the functor of $G$-locally analytic vectors; concretely, given a $G$-representation on a $\QQ_p$-Banach space $M$, we have $$\mt{R}^i_{G \hla}(M) = \varinjlim_{G' \le G}\mt{H}^i(G',\mathcal{C}^{\la}(G',M)),$$ where the direct limit ranges over open normal uniform subgroups of $G$. Alternatively, one may fix an open normal uniform $G_0 \le G$ and set $G_n = G_0^{p^n}$, and then it is also the case that 
$$ \mt{R}^i_{G \hla}(M) = \varinjlim_n\mt{H}^i(G_n,\mathcal{C}^{\an}(G_n,M)).$$ Here $\mc{C}^{\an}(G_n,M)$ is endowed with the action $g(f)(x) = g(f(g^{-1}x))$. We refer the reader to \cite{RJRC22}  for more details on the definitions of these derived functors. 
In the particular case $i = 0$, we have the subspaces of analytic vectors $M^{G_n\han}=\mt{H}^0(G_n,\mc{C}^{\an}(G_n,M))$. These can also be defined in an alternative way, which we now explain. Since $G_0$ is uniform, we can define a homeomorphism $c:G_0\cong \ZZ_p^d$, so that there are induced homeomorphisms $c|_{G_n}:G_n\cong (p^n\ZZ_p)^d$. Then $M^{G_n\han}$ can also be defined as the space of $m \in M$ such that there exists an expansion $g(m) = \sum_{\bf{k} \in \NN^d}c(g)^{\bf{k}}m_{\bf{k}}$ with $p^{n|\bf{k}|}m_{\bf{k}}\to 0$. Using this perspective one sees that $M^{G_n\han}$ is a Banach space with valuation given by $\val_{G_n}(m) = \inf_{\bf{k}}(n|{\bf{k}}|+\val_M(m_{{\bf{k}}})).$

\subsubsection{Derivations}
\label{subsubsection:derivations}

We often consider a linear extension of an action by derivations. These are defined in the following way. Let $R$ be a commutative $\QQ_p$-algebra, and let $M$ be an $R$-module. Suppose $M$ has an action of $\Lie\Gamma$ by $\QQ_p$-linear derivations. Given $\Theta \in R\otimes_{\QQ_p}\Lie\Gamma$, we can act with $\Theta=\sum b_i \otimes \partial_i$ on $M$ by setting $\Theta(x):=\sum_i b_i \partial_i(x)$. This action is well defined (independent of the specific writing $\Theta = \sum b_i \otimes \partial_i$).

\subsection{Acknowledgments}
We thank Pierre Colmez, Rustam Steingart and Yupeng Wang for useful discussions and correspondences. 
 Part of this work was first carried out while Gal Porat was visiting the University of Bordeaux (IMB). He would like to thank the institute for the hospitality.
We were assisted by ChatGPT 5.5-5.6 Pro for simplifying some arguments and for polishing the article.
Hui Gao is partially supported by the National Natural Science Foundation of China under agreement NSFC-12471011. 
 Gal Porat was supported
by the European Research Council (ERC 101161909). 
Léo Poyeton is partially supported by the French Agence Nationale de la Recherche under agreement ANR-25-CE40-4664.

\section{Vanishing of higher locally analytic vectors}
\label{section:vanishing_of_higher_la}

\subsection{The vanishing theorem}

 Recall our notation and conventions $\mathsection$\ref{subsubsection:p_adic_hodge_theory_notations} and $\mathsection$\ref{subsubsection:la_vectors_notations}. 
In this section we prove the following theorem. It is well known to the experts (see \cite[Remark 5.7]{porat2024locally} which follows the idea of \cite[Theorem 3.6.1]{pan2022locally}), but we include a proof here due to a lack of a suitable reference.

\begin{theorem}
\label{thm:higher_la_vanishing_k_infty}
Let $W$ be a finite free $\wh{K}_{\infty}$-semilinear representation of $\Gamma(K_\infty)$.
Then $\R^i_{\Gamma(K_\infty)\hla}(W) = 0$ for $i \ge 1$.
\end{theorem}

The following corollary will be used often in \S \ref{section:la_vecs_in_bdr}. It follows from the Theorem by d\'evissage.

\begin{corollary}
\label{cor:higher_la_vanishing_bdr}
Let $W$ be a finite length $\BdRplus(K_\infty)$-semilinear representation of $\Gamma(K_\infty)$. Then $\R^i_{\Gamma(K_\infty)\hla}(W) = 0$ for $i \ge 1$. 
\end{corollary}

\subsection{Reductions}

We first show that we are allowed to replace $K$ by a finite extension.
 
\begin{lemma}
\label{lem:reduction_finite_extension_vanishing_higher_la}
Let $K'$ be a finite extension of $K$ and let $K'_{\infty} = K'K_\infty$. Then Theorem \ref{thm:higher_la_vanishing_k_infty} is true for $K'_{\infty}/K'$ and $W' := \wh{K}_\infty'\otimes_{\wh{K}_\infty}W$ if and only if it is true for $K_\infty/K$ and $W$.
\end{lemma}

\begin{proof}

Assume first that $K'_\infty=K_\infty$, so that $W' = W$. Since $\Gamma(K_\infty/K')$ is an open subgroup of $\Gamma(K_\infty/K)$, we have $\R^i_{\Gamma(K_\infty/K')\hla}(W)=\R^i_{\Gamma(K_\infty/K)\hla}(W)$ and in this case the vanishing for $K'_\infty/K'$ or $K_\infty/K$ is clearly equivalent.

In the general case, by \cite[p. 97, Lemma 6]{serre1964corps}, there exists a finite extension $K_0'/K$ such that $K_0'$ is linearly disjoint from $K_\infty/K$ and $K_0'K_\infty = K_\infty'$. Let $E = K_0'K'$. Then by the case already proved, the theorem is equivalent for either of the towers $K_\infty'/K'$,  $K_\infty'/E$ or $K_\infty'/K_0'$. Hence we are free to replace $K'$ with $K_0'$.

Thus, we reduce to the case where $K'$ is linearly disjoint from $K_\infty$. We have then a natural isomorphism $\Gamma(K_\infty'/K')\cong \Gamma(K_\infty/K)$, and so open subgroups in these two groups   are identified. We may identify a system of open normal uniform subgroups $\Gamma_n(K_\infty/K)$ in $\Gamma(K_\infty/K)$ with a system of open normal uniform subgroups $\Gamma_n(K_\infty'/K')$ in $\Gamma(K_\infty'/K')$.
 Furthermore, we have
 $$W'= \wh{K}_\infty'\otimes_{\wh{K}_\infty}W=K' \otimes_K W,$$
 So there is a $\Gamma_n(K_\infty'/K') \cong \Gamma_n(K_\infty/K)$-equivariant isomorphism
 \begin{align*}\mc{C}^{\an}(\Gamma_n(K_\infty'/K'),W')&=\mc{C}^{\an}(\Gamma_n(K_\infty'/K'),W\otimes_K K')\\&=\mc{C}^{\an}(\Gamma_n(K_\infty/K),W)\otimes_K K'\\&\cong \mc{C}^{\an}(\Gamma_n(K_\infty/K),W)^{\oplus [K':K],}
 \end{align*}
 Hence, for every $i \ge 0$ we have
 $$
 \R^i_{\Gamma(K_\infty'/K')\hla}(W') \cong \R^i_{\Gamma(K_\infty/K)\hla}(W)^{\oplus[K':K]}$$
 and the result follows.\end{proof}

In what follows, we let $L_\infty = K_\infty K_\cyc$. The tower $L_\infty/K$ still fits into our framework, in the sense that $\Gamma(L_\infty/K)$ is an infinitely ramified $p$-adic Lie group.

\begin{proposition}
\label{prop:reduction_to_infinite_delta}
If $L_\infty/K_\infty$ is a finite extension then Theorem \ref{thm:higher_la_vanishing_k_infty} holds for $K_\infty/K$.
\end{proposition}

\begin{proof} By the previous lemma, it suffices to prove it for the tower $L_\infty/K$, which contains $K_\cyc$. The tower $L_\infty/K$ satisfies the Tate-Sen axioms (TS1)-(TS3) by \cite[Proposition 4.1.1]{berger2008familles}, \cite[{$\mathsection{14.1}$}]{brinon2009cmi}, and (TS4)  by \cite[Example 5.5]{porat2024locally}. Hence we conclude by \cite[Theorem C]{porat2024locally}.
\end{proof}

\begin{lemma} \label{lem:assume_zp_ext}
Suppose $L_\infty/K_\infty$ is infinite. Then for the purpose of proving Theorem \ref{thm:higher_la_vanishing_k_infty}, we may assume $\Gal(L_\infty/K_\infty)\cong \ZZ_p$. Writing $\gamma$ for a topological generator of $\Gal(L_\infty/K_\infty)$, we may also assume there exists an element $z\in \wh{L}_\infty^{\Gamma(L_\infty/K)\hla}$ with $\gamma(z) = z+1$.
\end{lemma}

\begin{proof} The cyclotomic character gives an isomorphism of $\Delta=\Gal(L
_\infty/K_\infty)$ with an open subgroup of $\ZZ_p^\times$. Replacing $K$ by a finite extension, as we may, forces $\Gal(K_\cyc/K)\cong\ZZ_p$, and so we may assume $\Delta \cong \ZZ_p$. By \cite[Lemme 3.6]{berger2016theorie}, after replacing $K$ by a finite extension (as we may) there exists an element $z\in \wh{L}_\infty^{\Gamma(L_\infty/K)\hla}$ with $g(z) = z+\log\chi_\cyc(g)$ for $g \in \Delta$. Let $\gamma$ be a generator of $\Delta$. Replacing $z$ with $z/\log\chi_\cyc(\gamma)$, we may assume $\gamma(z)=z+1$.
\end{proof}

\subsection{Proof of the theorem}

We now prove Theorem \ref{thm:higher_la_vanishing_k_infty}. In this subsection, we use some condensed mathematics to simplify cohomology computations; these techniques will not be used elsewhere in the article. All the objects in what follows will be considered in the sense explained in \cite{RJRC22}, \cite{jacinto2023solid}, using the fully faithful inclusion of Banach representations of compact $p$-adic Lie groups into the appropriate categories of solid representations in $\QQ_p$-modules, as explained in detail in loc. cit. 

We need the following lemma.
\begin{lemma}
Let $1 \to N \to G \to G/N \to 1$ be an exact sequence of compact $p$-adic Lie groups. 
Then the inflation functor $\mt{Inf}_{G,N}$ from complexes of $\QQ_p[G/N]$-modules to complexes of $\QQ_p[G]$-modules  (respectively the locally analytic inflation functor $\mt{Inf}^{\la}_{G,N}$ from complexes of locally analytic $G/N$-representations to complexes of locally analytic $G$-representations) is left adjoint to the functor of $N$-fixed points $C \mapsto \mt{RHom}_{\QQ_p[N]}(\QQ_p,C)$.
\end{lemma}
\begin{proof}
Consider first the continuous setting. The inflation functor $\mt{Inf}_{G,N}$ from $\QQ_p[G/N]$-complexes to $\QQ_p[G]$-complexes is given by restricting from $\QQ_p[G/N]$-complexes to $\QQ_p[G]$-complexes. We have the series of identities
\begin{align*}
\mt{RHom}_{\QQ_p[G]}(\mt{Inf}_{N}^G(C_1),C_2) 
&= \mt{RHom}_{\QQ_p[G/N]}(C_1,\mt{RHom}_{\QQ_{p}[G]}(\QQ_p[G/N],C_2))\\
&= \mt{RHom}_{\QQ_p[G/N]}(C_1,\mt{RHom}_{\QQ_p[N]}(\QQ_p,C_2)),
\end{align*}
where the last equality follows from Shapiro's lemma (\cite[Proposition 5.15]{RJRC22}). This proves the lemma in the continuous setting. In the locally analytic setting, one can argue similarly by restricting from $D^{h^+}(G/N,\QQ_p)$-complexes to $D^{h^+}(G,\QQ_p)$-complexes and using the $h^+$-analytic version of Shapiro's lemma of loc. cit.
\end{proof}

\begin{proposition} \label{prop: HS_spe_seq}
    
Let $1 \to N \to G \to G/N \to 1$ be an exact sequence of compact $p$-adic Lie groups. Let $C$ be a complex of solid $\QQ_p[G]$-modules. Then there is a natural quasi-isomorphism of complexes $$\mt{RHom}_{\QQ_p[N]}(\QQ_p,C^{\mt{R}G\hla}) = (\mt{RHom}_{\QQ_p[N]}(\QQ_p,C))^{\mt{R}(G/N)\hla}.$$ Here $\mt{R}G\hla$ and
${\mt{R}(G/N)\hla}$ denote derived locally analytic vectors as in \cite{RJRC22}.
\end{proposition}
\begin{proof}
For $\bullet \in \{G,G/N\}$, let  $D(\QQ_p[\bullet])$ (respectively $\mt{Rep}_{\QQ_p}^\la(\bullet)$) be the derived category of solid $\QQ_p[\bullet]$-modules (respectively the  subcategory of locally analytic $\bullet$-representations). Let $$\iota_\bullet:\mt{Rep}_{\QQ_p}^{\la}(\bullet)\to D(\QQ_p[\bullet])$$ be the fully faithful inclusion functor (see \cite[Remark 1.2.1(2)]{jacinto2023solid}).
We have the following commutative diagram:

\[\begin{tikzcd} [row sep=large, column sep=large]
   \mt{Rep}^{\la}(G/N) \arrow[r,"\mt{Inf}^\la_{G,N}"] \arrow[d,"\iota_{G/N}"]
    & \mt{Rep}^{\la}(G)\arrow[d,"\iota_{G}"] \\
      D(\QQ_p[G/N]) \arrow[r,"\mt{Inf}_{G,N}"] 
    & D(\QQ_p[G]).
\end{tikzcd}
\]

By the previous lemma, the functors $\mt{Inf}_{G,N}, \mt{Inf}^{\la}_{G,N}$ have as right adjoint the $N$-fixed points functor $C\mapsto \mt{RHom}_{\QQ_p[N]}(\QQ_p,C)$. By \cite[Theorem A, Remark 1.2.1(2)]{jacinto2023solid}, for $\bullet \in \{G,G/N\}$ the  functor $\iota_\bullet$ has as right adjoint the functor $\mt{R}(\bullet)\hla$. It follows that $\iota_G\circ \mt{Inf}_{G,N}^\la$ has right adjoint the functor
$$C\mapsto\mt{RHom}_{\QQ_p[N]}(\QQ_p,C^{\mt{R}G\hla})$$
while $\mt{Inf}_{G,N}\circ \iota_{G/N}$ has as right adjoint the functor
$$C \mapsto (\mt{RHom}_{\QQ_p[N]}(\QQ_p,C))^{\mt{R}(G/N)\hla}.$$
Since $\iota_G\circ \mt{Inf}_{G,N}^\la=\mt{Inf}_{G,N}\circ \iota_{G/N},$ these two right adjoints are naturally identified, which concludes the proof.
\end{proof}

Using the reductions of the previous subsection (Proposition \ref{prop:reduction_to_infinite_delta} and Lemma \ref{lem:assume_zp_ext}), we may assume $L_\infty/K_\infty$ is an infinite extension, with $\Delta=\Gal(L_\infty/K_\infty)=\gamma^{\ZZ_p}$ and that there exists $z \in \wh{L}_\infty^\la$ with $\gamma(z) = z+ 1$. Let $\nabla=\nabla_\gamma$ be the generator of $\Lie \Delta $ corresponding to $\gamma$, we have $\nabla(z)=1$. The map $\Gamma(L_\infty/K)\to\Gamma(K_\infty/K)$ has image $\Gal(K_\infty/K)$. We may and do assume $K_\cyc\cap K_\infty =K$ again by performing a finite extension of $K$. For the rest of the proof, write $\Gamma(L_\infty):= \Gamma(L_\infty/K)$ and $\Gamma(K_\infty):= \Gamma(K_\infty/K)$ to lighten the notation.
We let $W(K_\infty):= W$ and $W(L_\infty) := \wh{L}_\infty \otimes_{\wh{K}_\infty}W(K_\infty)$, which is a $\Gamma(L_\infty)$-semilinear representation. We can write the following chain of quasi-isomorphisms:
\begin{align*}
W(K_\infty)^{\mt{R}\Gamma(K_\infty)\hla} &= \mt{RHom}_{\QQ_p[\Delta]}(\QQ_p,W(L_\infty))^{\mt{R}\Gamma(K_\infty)\hla}\\
&= \mt{RHom}_{\QQ_p[\Delta]}(\QQ_p,W(L_\infty)^{\mt{R}\Gamma(L_\infty)\hla}) \\
&=\mt{RHom}_{\QQ_p[\Delta]}(\QQ_p,W(L_\infty)^{\Gamma(L_\infty)\hla}).
\end{align*}
We explain the equalities above.
Since the extension $\wh{L}_\infty/\wh{K}_\infty$ is pro-étale ($\wh{K}_\infty$ is perfectoid), we know by the almost purity theorem that $\mt{H}_{\mt{cont}}^i(\Delta,W(L_\infty))=0$ for $i > 0$ and is equal to $W(K_\infty)$ for $i=0$, which gives the first equality. The second equality follows from Proposition \ref{prop: HS_spe_seq}. For the third equality, use that as in the proof of Proposition \ref{prop:reduction_to_infinite_delta}, we know that $\R^i_{\Gamma(L_\infty)\hla}(W(L_\infty))=0$ for $i>0$.
Hence, the proof of the theorem reduces to the following key computation.

\begin{lemma} \label{lem:delta_no_higher_coho}
 
We have $\mt{H}^i_{\mt{cont}}(\Delta,W(L_\infty)^{\Gamma(L_\infty)\hla}) = 0$ for $i > 0$.
\end{lemma}

\begin{proof}
In the proof of \cite[Théorème 3.4]{berger2016theorie} it is shown that $$W(L_\infty)^{\Gamma(L_\infty)\hla}=\wh{L}_{\infty}^\la\otimes_{\Kinftyla}W(K_\infty)^{\Gamma(K_\infty)\hla}.$$ As $\Delta$ acts trivially on $W(K_\infty)^{\Gamma(K_\infty)\hla}$, we reduce to the case $W(L_\infty) = \wh{L}_{\infty}$. 
Since $\Delta\cong \ZZ_p$, the only nontrivial case is $i=1$. Let $\Gamma_n$ be the usual sequence of open normal uniform subgroups for $\Gamma(L_\infty)$. Then we have $\Linftyla = \varinjlim \wh{L}_\infty^{\Gamma_n\han}$. The complex computing continuous cohomology of $\Delta$ is $\mt{RHom}_{\QQ_p[\Delta]}(\QQ_p,\Linftyla),$ and since $\QQ_p$ is a compact object in the category of solid $\QQ_p[\Delta]$-modules (it is the cokernel of a map $\QQ_p[\Delta]\to \QQ_p[\Delta]$), we learn that 
$\mt{H}^1_{\mt{cont}}(\Delta,\Linftyla)=\varinjlim_n\mt{H}^1_{\mt{cont}}(\Delta,\wh{L}_\infty^{\Gamma_n\han}).$ Hence, we can use Lazard's theorem (\cite[Theorem 1.7]{RJRC22}) to get 
$\mt{H}^1_{\mt{cont}}(\Delta,\wh{L}_\infty^{\Gamma_n\han}) = \mt{H}^1_{\Lie \Delta}(\wh{L}_\infty^{\Gamma_n\han})^{\Delta=1}.$

We therefore reduce to showing that the direct limit $\varinjlim_n \mt{H}^1_{\Lie \Delta}(\wh{L}_\infty^{\Gamma_n\han})$ is zero. 
We are going to do this by showing that for every $n$ there exists $N \ge n$ such that the map $\mt{H}^1_{\Lie \Delta}(\wh{L}_\infty^{\Gamma_n\han}) \to \mt{H}^1_{\Lie \Delta}(\wh{L}_\infty^{\Gamma_N\han})$ is zero. Equivalently, we show that for every $x \in \wh{L}_\infty^{\Gamma_n\han}$ there exists $y \in \wh{L}_\infty^{\Gamma_N\han}$ such that $\nabla(y)=x$.

We are going to define $y$ as a power series. For $m \ge 1$, let $\val_{\Gamma_m}$ denote the valuation of the Banach space $\wh{L}_\infty^{\Gamma_m\han}$. 
By \cite[Lemme 2.6]{berger2016theorie},
we have $\nabla( \wh{L}_\infty^{\Gamma_n\han}) \subset  \wh{L}_\infty^{\Gamma_n\han}$ and there exists some constant $C_n$ (possibly negative) such that for $x \in \wh{L}_\infty^{\Gamma_n\han}$
we have $$\val_{\Gamma_{n}}(\nabla(x)) \ge C_n + \val_{\Gamma_{n}}(x).$$




Let $z_0 \in L_\infty$ such that the $p$-adic valuation $a=\val_{\wh{L}_\infty}(z-z_0)$ satisfies$$a>1/(p-1)+\max(0,-C_n + 1).$$ Since $z_0$ is smooth, we still have $\nabla(z-z_0)=1$, so we may replace $z-z_0$ with $z$. We may choose $N\ge n$ sufficiently large such that $z \in \wh{L}_\infty^{\Gamma_{N\han}}$ and $\val_{\Gamma_{N\han}}(z) = \val_{\wh{L}_\infty}(z) = a$.

For $N$, we thus have 
\begin{align*}
\val_{\Gamma_{N\han}}(\nabla^i(x)\cdot \frac{z^{i+1}}{(i+1)!}) &\ge \val_{\Gamma_{N\han}}(\nabla^i(x)) + \val_{\Gamma_{N\han}}(z^{i+1}/(i+1)!)\\
&\ge\val_{\Gamma_{n\han}}(\nabla^i(x)) + \val_{\wh{L}_\infty}(z^{i+1}/(i+1)!)\\
&\ge \val_{\Gamma_{n\han}}(x) + iC_n + (i+1)a -(i+1)/(p-1)\\
&\ge \val_{\Gamma_{n\han}}(x) + i,
\end{align*}
where the last step is because of the condition on $a$.

Hence, we have
$\val_{\Gamma_{N\han}}(\nabla^i(x)\cdot {z^{i+1}}/{(i+1)!}) \to \infty$
as $i \to \infty$. Thus, the power series
$$y = \sum_{i=0}(-1)^i\nabla^i(x)\frac{z^{i+1}}{(i+1)!}$$
converges in $\wh{L}_\infty^{\Gamma_N\han}$. A straightforward check shows that $\nabla(y) = x$, which concludes the proof.\end{proof}

\section{Sen operators}
\label{section:sen_operators}

Recall our notation from $\mathsection$\ref{subsection:notation_and_conventions}. In \S \ref{subsec:sen_op_def} and \S \ref{subsec:adj_rep_sen},  we  recall the Sen operator for the extension $K_\infty/K$, following \cite{sen1980continuous} and \cite{berger2016theorie}.
In \S \ref{subsec:rationality}, we study the field of rationality of the Sen operator.

We start by recording an interesting fact (well known to the experts) which shows that $p$-adic Lie extensions $\kinfty/K$ are abundant. This fact serves as a justification of our setup, and will not be used elsewhere in the article.

\begin{lemma} \label{lem:pro_p_as_Gal_gp}
Let $G$ be a topologically finitely generated pro-$p$ group.
Then there exist a finite extension $K/\Qp$ and a Galois extension $L/K$
such that $\Gal(L/K)\cong G$ as topological groups.
\end{lemma}
\begin{proof} 
By Theorems 7.5.11-7.5.12 of \cite{NSW}, one can choose $K$ such that $G_K$ surjects onto a free pro-p group with arbitrarily large rank, which then surjects onto $G$.
\end{proof}

\subsection{The Sen operator of the tower $K_\infty/K$} \label{subsec:sen_op_def}

 In this subsection we recall the construction of the Sen operator for $V$ a finite-dimensional $\QQ_{p}$-linear
representation of $G_{K}$. 

Let $V_{\C}:=\C\otimes_{\QQ_p} V$, considered as a \emph{semilinear}
representation of $G_{K}$. We have the $K_{\cyc}$-vector space
\[
\D_{\Sen}(V):=(V_{C})^{\Gal(\overline{K}/K_\cyc),\Gamma(K_\cyc)\hfin},
\]
that is, the $K_{\cyc}$-subspace of elements in $V_{C}$ which are
fixed under $H_{\cyc}$, and which lie in a $\Gamma(K_\cyc)$-stable
finite-dimensional $K$-vector space. It is $\Gamma(K_\cyc)$-stable,
and hence $G_{K}$-stable. In \cite{sen1980continuous}, Sen proves that the natural
map
\[
\D_{\Sen}(V)_{\C}:=\C\otimes_{K_{\cyc}}\D_{\Sen}(V)\rightarrow V_{\C}
\]
is an isomorphism of semilinear $G_{K}$-representations.

With this given, one defines the Sen operator $\Theta(V)\in\End_{K_{\cyc}}(\D_{\Sen}(V))$
(a linear operator!) as follows. Choose $\gamma\in\Gamma(K_\cyc)$
which is sufficiently close to 1. Then 
\[
\Theta(V):=\log\gamma/\log\chi_{\cyc}(\gamma).
\]
This is independent of any choice. There is an inclusion
\[
\End_{K_{\cyc}}(\D_{\Sen}(V))\hookrightarrow\End_{\C}(\D_{\Sen}(V)_{\C})\cong\End_{\C}(V_{\C})=\C\otimes_{\QQ_p}\End_{\QQ_{p}}(V)
\]
which allows us to consider $\Theta(V)$ as an element of $\End_{\C}(V_{\C})$.
In fact, we have the following theorem of Sen (\cite[Theorem 11]{sen1980continuous}),
which shows $\Theta(V)$ belongs to a smaller subspace.
\begin{theorem}
\label{thm:sen_op_in_lie_alg}
Let $\rho:G_{K}\rightarrow\GL_{\QQ_{p}}(V)$ denote the linear representation
defining $V$. Then we have 
\[
\Theta(V)\in \C \otimes_{\QQ_p} \Lie(\rho(G_{K}))\subset \C\otimes_{\QQ_p}\End_{\QQ_{p}}(V).
\]
\end{theorem}

Given a faithful representation $V$ of $\Gamma(K_\infty)$ (which exists by Ado's theorem \cite[{$\mathsection{7.3}$}]{bourbaki1975lie}), we set (implicitly
using Theorem \ref{thm:sen_op_in_lie_alg})
\[
\Theta(K_{\infty}/K,V):=\Theta(V)\in \C\otimes_{\QQ_p}\Lie\Gamma(K_\infty).
\]

\begin{proposition}
\label{prop:sen_operator_independant_of_faithful_rep_and_nonzero}
The element $\Theta(K_{\infty}/K,V)$ is independent of $V$ and depends on $K$ only up to a finite extension. It is nonzero.
\end{proposition}
\begin{proof}
See \cite[Theorem 12, (a)-(b)]{sen1980continuous}. To show it is nonzero, one can use for example the Corollary to \cite[Theorem 11]{sen1980continuous}. 
\end{proof}
This allows us to give the following definition.

\begin{definition}[The Sen operator of $K_\infty$]
\label{def:definition_of_sen_operator}
The Sen operator of $K_\infty$ is the unique element $$\Theta(K_\infty)\in \C \otimes_{\QQ_p}\Lie\Gamma(K_\infty)$$ satisfying $\Theta(K_\infty)=\Theta(K_\infty/K,V)$ for any faithful representation $V$ of $\Gamma(K_\infty)$.    
\end{definition}

The following compatibility will be useful.

\begin{lemma}
\label{lemma:sen_operator_compatibility}
  Let $K_\infty \subset L_\infty$  be two $p$-adic Lie extensions  of $K$. 
Then the natural morphism 
$$\C \otimes_{\QQ_p} \Lie\Gamma(L_\infty) \to \C \otimes_{\QQ_p} \Lie\Gamma(K_\infty), $$ which is induced by the quotient map  $\Gamma(L_\infty)  \to \Gamma(K_\infty)$, 
maps $\Theta(L_\infty)$ to $\Theta(K_\infty).$ 
\end{lemma}
\begin{proof}Let \(V(K_\infty)\) be a faithful representation of \(\Gamma(K_\infty)\), and let
\(V(L_\infty)\) be a faithful representation of \(\Gamma(L_\infty)\).
 By assumption, there is a morphism $
\Gamma(L_\infty)\to \Gamma(K_\infty)$.
Precomposing this morphism, we may consider $V(K_\infty)$ as a representation of $\Gamma(L_\infty)$. The representation $V(L_\infty) \oplus V(K_\infty)$ is then a faithful representation of $\Gamma(L_\infty)$.
We have a commutative diagram
\[
\begin{tikzcd}
\Gamma(L_\infty)
  \arrow[r]
  \arrow[d]
&
\Gamma(K_\infty)
  \arrow[d]
\\
\operatorname{GL}_{\mathbf Q_p}
\bigl(V(L_\infty))\times\GL_{\QQ_p}( V(K_\infty)\bigr)
  \arrow[r]
&
\operatorname{GL}_{\mathbf Q_p}\bigl(V(K_\infty)\bigr),
\end{tikzcd}
\]
where the bottom horizontal map is induced by the projection
$
V(L_\infty)\oplus V(K_\infty)\to V(K_\infty).
$
Differentiating and tensoring with $\C$, we obtain a commutative diagram
\[
\begin{tikzcd}
\C\otimes_{\mathbf Q_p}\Lie\Gamma(L_\infty)
  \arrow[r]
  \arrow[d]
&
\C\otimes_{\mathbf Q_p}\Lie\Gamma(K_\infty)
  \arrow[d]
\\
\End_{\C}\!\bigl(
V(L_\infty)_\C)\oplus\End_\C( V(K_\infty)_\C
\bigr)
  \arrow[r]
&
\End_{\C}\!\bigl(V(K_\infty)_\C\bigr).
\end{tikzcd}
\]

It therefore suffices to show that the bottom horizontal map sends the Sen
operator of \(V(L_\infty)\oplus V(K_\infty)\) to the Sen operator of
\(V(K_\infty)\). This follows from the equality $\Theta(V(L_\infty)\oplus V(K_\infty)) = \Theta(V(L_\infty)) \oplus \Theta(V(K_\infty)$ (see Remark (4) after \cite[Theorem 4]{sen1980continuous}).\end{proof}

\subsection{The adjoint $p$-adic Galois representation of $K_\infty$} \label{subsec:adj_rep_sen}

Continue with the notation of the previous subsection. We now introduce one of the key objects of this paper: the $p$-adic representation $\Lie\Gamma(K_\infty)$, also called the adjoint representation of $K_\infty$. We give it the structure of a $p$-adic representation by precomposing the adjoint representation of $\Gamma(K_\infty)$ with the projection $G_K \twoheadrightarrow \Gamma(K_\infty)$. We may consider $p$-adic Hodge theoretic invariants of $\Lie\Gamma(K_\infty)$ as in classical $p$-adic Hodge theory. Two key players for us will be the $0$-Hodge Tate weight space
$$\D_{\C}(\Lie\Gamma(K_\infty)) = (\C \otimes_{\QQ_p} \Lie\Gamma(K_\infty))^{G_K}$$
and the positive de Rham space
$$\DdRplus(\Lie\Gamma(K_\infty))=(\BdRplus\otimes_{\QQ_p}\Lie\Gamma(K_\infty))^{G_K}.$$

 Because of the definition of the $G_K$ action on $\Lie\Gamma(K_\infty)$, item (1) of the following theorem can be interpreted as saying that as a $\C$-linear operator,  $\Theta(K_{\infty})$ is $G_K$-equivariant.

\begin{theorem} \label{thm:sen_is_analytic}
\hfill
\begin{enumerate}
    \item We have $\Theta(K_{\infty})   \in \D_{\C}(\Lie\Gamma(K_\infty))$.
    \item We have $\Theta(K_\infty) \in \widehat{K}_{\infty}^\la\otimes_{\QQ_p}\Lie\Gamma(K_\infty).$
\end{enumerate}  

\end{theorem}

\begin{proof}
Item (1) is \cite[Theorem 12 (d)]{sen1980continuous}. For item (2), as $\Theta(K_\infty)$ is a $\gammak$-fixed point of 
$\hatkinfty \otimes_\qp \Lie \gammak$, it is in particular   a locally analytic vector; we can conclude using Lemma \ref{lem:basis_lav}, as $\Lie \gammak$ is automatically locally analytic. 
\end{proof}

\begin{remark}
Theorem \ref{thm:sen_is_analytic}(1) gives a nice way to see the functoriality of the association $K_\infty \mapsto \Theta(K_\infty)$. Let $\mt{LieExt}_K$ be the category whose objects are $p$-adic Lie extensions $K_\infty/K$, with morphisms  $K_\infty/K\to K_\infty'/K$ being the same as a morphism $K_\infty \to K_\infty'$ which is the identity on $K$. Then the associations $K_\infty/K\mapsto \Lie\Gamma(K_\infty)$ and $K_\infty/K \mapsto \D_{\C}(\Lie(\Gamma(K_\infty)))$ are contravariant functors, and $\D_{\C}(\Lie(\Gamma(K_\infty'))) \to \D_{\C}(\Lie(\Gamma(K_\infty)))$ maps $\Theta(K_\infty')$ to $\Theta(K_\infty)$, by Lemma \ref{lemma:sen_operator_compatibility}. 
\end{remark}

We shall need the following important result due to Berger and Colmez. Recall $\mathsection$\ref{subsubsection:derivations} for our definitions for the action of $\wh{K}_\infty^\la\otimes_{\QQ_p}\Lie\Gamma(K_\infty)$ on $\wh{K}_\infty^\la$. 

\begin{proposition}
\label{prop:sen_operator_kills_la_vecs}
The Sen operator $\Theta(K_\infty) \in \widehat{K}_{\infty}^\la\otimes_{\QQ_p}\Lie\Gamma(K_\infty)$ acts by zero on $\widehat{K}_{\infty}^\la$.
\end{proposition}
\begin{proof}
To show this, consider the proof of \cite[Théorème 6.1]{berger2016theorie}. It is shown there that there is an isomorphism
$$\wh{K}_\infty(\mu_{p^m})\otimes_{K_n} \wh{K}_\infty^{\Gamma_n\han} \xrightarrow[]{\cong}\mc{C}^{\an}(\Gamma_n,\wh{K}_\infty)^{D=0},$$
where for $\phi \in \mc{C}^{\an}(\Gamma_n,\wh{K}_\infty)$, we have $D(\phi)(g) = \lim_{t \to 0}\frac{\phi(e^{t\mathfrak{a}}g)-\phi(g)}{t}$, for some specific ${\mathfrak{a}} \in\wh{K}_\infty(\mu_{p^m})\otimes_{\QQ_p}\Lie \Gamma(K_\infty)$. In \cite[Proposition 6.3]{berger2016theorie} it is proved that ${\mathfrak{a}}$ is the Sen operator of the representation $V_1$ of $\Gamma(K_\infty)$ which is given as follows: $V_1$ is the space of functions $\Gamma(K_\infty) \to \QQ_p$ which are restrictions of linear functions from $\Mat_{d}(\QQ_p)$ to $\Gamma(K_\infty)$, and where $\Gamma(K_\infty)$ acts on $V_1$ by the formula $(g\phi)(x)=\phi(g^{-1}(x))$. The representation $V_1$ of $\Gamma(K_\infty)$ is faithful, and so ${\mathfrak{a}}=\Theta(K_\infty)$. It thus suffices to explain why $\mathfrak{a}$ kills $\wh{K}_{\infty}^{\Gamma_n \han}.$
In the isomorphism mentioned above, an element $x \in \wh{K}_\infty^{\Gamma_n\han}$ is mapped to the function $\phi$ which is given by $\phi(g) = g(x)$. In particular $\phi(1)=x$. Hence as $D(\phi)(1)=0$, we get
$$0=D(\phi)(1) = \lim_{t \to 0}\frac{\phi(e^{t\mathfrak{a}})-\phi(1)}{t}=\lim_{t \to 0}\frac{e^{t\mathfrak{a}}(x)-x}{t}=\mathfrak{a}(x),$$
as required.
\end{proof}

\subsection{The field of rationality of $\Theta(\kinfty)$} \label{subsec:rationality}

In this subsection, we discuss the smallest field of $\C$ over which $\Theta(K_\infty)$ is defined. 

\begin{definition}[The field of rationality $E(K_\infty)$]
\label{def:rationality}
 We let $E(K_\infty)$ be the smallest subfield of $\C$ containing $K$ such that $\Theta(K_\infty) \in E(K_\infty)\otimes_{\QQ_p}\Lie\Gamma(K_\infty).$ Equivalently, let $\{\partial_i\}_i$ be a $\QQ_p$-basis 
      of $\Lie\Gamma(K_\infty)$. Writing $\Theta(K_\infty) = \sum a_i \otimes \partial_i \in \C \otimes_{\QQ_p} \Lie\Gamma(K_\infty)$, one sees that the field $E(K_\infty)$ generated over $K$ by the $a_i$ is independent of the choice of basis $\{\partial_i\}_i$. 
\end{definition}

In the following theorem, write $\ad(K_\infty):\Gamma(K_\infty) \to \GL(\Lie(\Gamma(K_\infty))$ for the adjoint representation.  Then $\widehat{K}_{\infty}^{\ker(\ad(K_\infty))} \subset \widehat{K}_\infty$ is a closed subfield.

\begin{theorem}
\label{thm:field_of_def_of_sen_op_is_gamma_finite}
The field $E(K_\infty)$ is contained in $\wh{K}_\infty^{\ker(\ad(K_\infty)),\Gamma(K_\infty)\hfin},$ the subring of elements in $\wh{K}_\infty^{\ker(\ad(K_\infty))}$ which have $\Gamma(K_\infty)$-orbit of finite $\QQ_p$-dimension.

\end{theorem}

\begin{proof}
Let $\pi:G_K\to \Gamma(K_\infty)$ be the projection. The subgroup $\ker(\ad(K_\infty)\circ\pi)$ fixes $\Lie\Gamma(K_\infty)$ and $\Theta(K_\infty)$ (the latter being fixed by all of $G_K$). Hence all of the $a_i$ in Definition  \ref{def:rationality} are fixed by this subgroup also. It follows that the field $E(K_\infty)$ is contained in the corresponding fixed field $\wh{K}_\infty^{\ker\ad(K_\infty)}$. Now express $\Theta(K_{\infty})=\sum a_{i} \otimes \partial_i$ with
$a_{i}\in \wh{K}_\infty^{\ker\ad(K_\infty)}$ such that the $\partial_{i}$ form a basis of $\Lie\Gamma(K_\infty)$. Then for every $g \in \Gamma(K_\infty),$
\begin{align}
\label{eq:ad_identity}
\sum a_i \otimes \partial_i=\sum g(a_{i})\otimes \mathrm{ad}(g)(\partial_{i}).
\end{align}
Write $h_{ij}(g)$ for the coefficients of the matrix of $\ad(g^{-1})^t$ in the basis $\{{\partial_i}\}_{i=1,...,d}.$ They lie in $\QQ_p$. The equation (\ref{eq:ad_identity}) implies that we have $g(a_i) = \sum_j h_{ij}(g)a_j$, 
which shows that the $\Gamma(K_\infty)$-orbit of the $a_i$ is finite-dimensional over $\QQ_p.$
\end{proof}

\begin{corollary}
\label{cor:sen_op_rationality}  If $\Gamma(K_\infty)$ is abelian, then $\Theta(K_\infty) \in K \otimes_{\QQ_p} \Lie\Gamma(K_\infty)$.
\end{corollary}
  

When $\Lie\Gamma(K_\infty)$ is Hodge-Tate we can get finer results for $E(K_\infty)$. We record these for the interest of the reader although they will not be used in the rest of the article.

\begin{proposition} \label{prop:trans_deg_sen}
Suppose  $\Lie\Gamma(K_\infty)$ is a Hodge-Tate representation. Then the following three numbers are equal:
\begin{enumerate}
\item
$\mt{trdeg}_{K}E(K_\infty)$.

\item $\dim\Gamma(K_\infty)-\dim\mt{Z}_{\Lie\Gamma(K_\infty)}(\Theta(K_\infty)),$ where $\mt{Z}_{\Lie\Gamma(K_\infty)}(\Theta(K_\infty))$ is the centralizer of $\Theta(K_\infty)$ in $\C\otimes_{\QQ_p} \Lie\Gamma(K_\infty)$. 

\item $\dim\Gamma(K_\infty)-\dim_{K}\D_{\C}(\Lie\Gamma(K_\infty)).$
\end{enumerate}
\end{proposition}
\begin{proof}
The equality between (1) and (2) is \cite[{$\mathsection{1.5}$}, Theorem 3]{serre1979groupes}. The equality between (2) and (3) follows from the fact that $\D_{\C}(\Lie\Gamma(K_\infty))$ is the space where $\Theta(\Lie\Gamma(K_\infty))$ acts by zero, and the observation that as an element of $\End_{\C}(\Lie(\Gamma(K_\infty)),$ the Sen operator $\Theta(\Lie\Gamma(K_\infty)) $ is none other than $ \ad(\Theta(K_\infty)).$ \end{proof}

\begin{corollary}
Suppose $\Lie\Gamma(K_\infty)$ is a Hodge-Tate representation. Then the Sen operator is defined over an algebraic extension of $K$ if and only if $\Lie\Gamma(K_\infty)$ is abelian. 
\end{corollary}
\begin{proof}
The previous proposition shows that $E(K_\infty)\subset \bar{K}$ if and only if $\Theta(K_\infty)$ is central in $\C\otimes_{\QQ_p} \Lie\Gamma(K_\infty)$. It therefore suffices to show that this condition implies that $\Lie\Gamma(K_\infty)$ is abelian. Let $\mt{Z}(K_\infty)$ be the center of $\Gamma(K_\infty)$. Possibly replacing $\Gamma(K_\infty)$ with an open subgroup, as we may do, $\mt{Z}(K_\infty)$ becomes a $p$-adic Lie group with Lie algebra $\mt{Z}(\Lie\Gamma(K_\infty))$. Let $\mt{P}\Gamma(K_\infty) = \Gamma(K_\infty)/\mt{Z}(K_\infty)$. Assume by contradiction that $\Lie\Gamma(K_\infty)$ is not abelian. Then we have $\mt{P}\Gamma(K_\infty) = \Gamma(K_\infty')$ for some $p$-adic Lie extension $K_\infty'$ of $K$. We have an exact sequence
$$0 \to \mt{Z}(\Lie\Gamma(K_\infty)) \to \Lie\Gamma(K_\infty) \to \Lie\mt{P}\Gamma(K_\infty) \to 0.$$
By the second paragraph of \cite[Chapter I, {$\mathsection{3}$}, no. 8]{bourbaki1975lie}, we have $$\mt{Z}(\C\otimes_{\QQ_p}\mt\Lie\Gamma(K_\infty)) =  \C\otimes_{\QQ_p}\mt{Z}(\Lie\Gamma(K_\infty)).$$ Hence by assumption we have $$\Theta(K_\infty) \in  \C\otimes_{\QQ_p}\mt{Z}(\Lie\Gamma(K_\infty)),$$ so the map $$\C\otimes_{\QQ_p}\Lie\Gamma(K_\infty) \to \C\otimes_{\QQ_p}\Lie\mt{P}\Gamma(K_\infty)$$ sends $\Theta(K_\infty)$ to $0$. But by Lemma \ref{lemma:sen_operator_compatibility}, this is also equal to $\Theta(K_\infty')$, which is nonzero by Proposition \ref{prop:sen_operator_independant_of_faithful_rep_and_nonzero}.
\end{proof}

\section{Orientations} \label{sec:ori}

In this section we introduce the concept of an orientation, which is a certain Galois-equivariant  lift of a Sen operator. As explained in the introduction, this lift will lead to canonical  descriptions of $\bdrpluskinftypa$ in   \S \ref{section:la_vecs_in_bdr}.


\subsection{Definition of orientations}


\begin{definition} \label{def:orientation}
    Recall by Theorem \ref{thm:sen_is_analytic}, we have
\[ \Theta(K_{\infty})   \in \D_{\C}(\Lie\Gamma(K_\infty)) = (\C \otimes_{\QQ_p} \Lie\Gamma(K_\infty))^{G_K} \]
By an orientation of $K_\infty$, we mean a lift of $\Theta(K_\infty)$ to an element \[ \nabla(K_\infty) \in \DdRplus(\Lie\Gamma(K_\infty)) =(\BdRplus\otimes_{\QQ_p}\Lie\Gamma(K_\infty))^{G_K}\]
We say $K_\infty$ is orientable if there exists an orientation of $K_\infty$. We say $K_\infty$ is canonically orientable if a unique orientation exists.
\end{definition}

\begin{lemma}
    If an orientation $\nabla(K_\infty)$ exists, then necessarily
    \[ \nabla(K_\infty) \in  (\BdRpluspa\otimes_{\QQ_p}\Lie\Gamma(K_\infty))^{\gammak}\]
\end{lemma}
\begin{proof}
    This is similar to  Theorem \ref{thm:sen_is_analytic}: a $\gammak$-fixed point is necessarily pro-analytic.
\end{proof}

\begin{notation} \label{nota:DC_DDR}
    We set up some notation which will be convenient in this section. Let $V_{\Sen}(K_\infty)$ be the $K$-span of $\Theta(K_\infty)$ in $\D_{\C}(\Lie\Gamma(K_\infty))$. Let $V_{\dR}^+(K_\infty)$ be the inverse image of $V_{\Sen}(K_\infty)$ under the natural map $\DdRplus(\Lie\Gamma(K_\infty)) \to \D_{\C}(\Lie\Gamma(K_\infty)).$ Thus $\kinfty$ is orientable if and only if the map $V_{\dR}^+(K_\infty) \to V_{\Sen}(K_\infty)$ is non-zero, and canonically orientable  if and only if the map $V_{\dR}^+(K_\infty) \to V_{\Sen}(K_\infty)$ is an isomorphism.
\end{notation}

\subsection{General criteria}
 
We start with an easy sufficient condition for orientability.

\begin{proposition}[de Rham extensions are orientable]  
\label{prop:de_rham_extensions_are_orientable}
Suppose $\Lie\Gamma(K_\infty)$ is de Rham as a representation. Then $K_\infty$ is orientable. It is canonically orientable if and only if  $\DdRplus(\Lie\Gamma(K_\infty)(1))=0$. 
\end{proposition}
\begin{proof}
The map $\DdRplus(\Lie\Gamma(K_\infty))\to \D_{\C}(\Lie\Gamma(K_\infty))$ identifies with the surjection   $$\Fil^0\D_{\dR}(\Lie\Gamma(K_\infty)) \onto \gr^0\D_{\dR}(\Lie\Gamma(K_\infty))$$ and has kernel $\DdRplus(\Lie\Gamma(K_\infty)(1)).$
\end{proof}

\begin{remark} \label{rem:future_criterion}
Write $V = \Lie\Gamma(K_\infty)$. We have an exact sequence
\begin{equation} \label{eq:seq_for_ori}
0\rightarrow\D_{\dR}^{+}(V(1))\rightarrow\D_{\dR}^{+}(V)\rightarrow\D_{\C}(V)\xlongrightarrow{\delta} \mt{H}^{1}(G_{K},\B_{\dR}^{+}\otimes V(1)).
\end{equation} 
The tower $K_\infty$ is orientable if and only if $\delta(\Theta(K_\infty))=0$. 
 Thus the image $\delta(\Theta(K_\infty))$  is the precise obstruction for orientability.
 However, this obstruction is in general difficult to compute directly; we shall see in Theorem \ref{thm:characterization_of_orientability} below an (equivalent) characterization of orientability (the proof will use results in \S \ref{section:la_vecs_in_bdr}).
\end{remark}

To give a useful characterization for $K_\infty$ being canonically orientable, we  recall a classical lemma.

\begin{lemma} \label{lem:no_dR_period}
Let $V$ be a $p$-adic Galois representation of $G_K$. The following are equivalent.

\begin{enumerate}
\item The representation $V$ has no generalized Hodge-Tate weights in $\ZZ_{\le 0}$.

\item We have  $\mt{H}^i(G_K,\BdRplus\otimes_{\QQ_p} V) = 0$ for $i = 0, 1$.

\item We have  $\D_{\dR}^+(V)= 0$.

\end{enumerate}

\end{lemma}

\begin{proof}
First we show that item (1) implies item (2). For $n \ge 0$, we have exact sequences
$$ 0\to t^{n+1}\BdRplus \to t^{n}\BdRplus \to \C(n) \to 0.$$
Tensoring with $V$ and taking $G_K$-cohomology, we get that the natural maps $$\mt{H}^i(G_K,t^{n+1}\BdRplus\otimes_{\QQ_p} V) \to \mt{H}^i(G_K,t^{n}\BdRplus\otimes_{\QQ_p} V)$$ for $i = 0,1$
 are isomorphisms. Hence we conclude by $t$-adic continuity.

Clearly item (2) implies item (3).
 To conclude, we show that the negation of item (1) implies the negation of item (3). Suppose then that $V$ has a generalized Hodge-Tate weight in $\ZZ_{\le 0}$, and let $n$ be the minimal integer. Observe that $V(-n+1)$ has no generalized Hodge-Tate weights in $\ZZ_{\le 0}$, so by the direction already established, we have $\mt{H}^i(G_K,\BdRplus\otimes_{\QQ_p} V(-n+1)) = 0$ for $i = 0, 1$. Hence $\D_{\dR}^+(V(-n)) \cong \D_{\C}(V(-n)),$ so $\D_{\dR}^+(V(-n)) $ is nonzero. But $\D_{\dR}^+(V(-n)) = \Fil^{-n}\D_{\dR}^+(V)$ is contained in $\D_{\dR}^+(V)$, so $\D_{\dR}^+(V)$ itself is nonzero.
\end{proof}

We deduce the following characterization of canonically orientable extensions:

\begin{theorem}[Characterization of canonically orientable $K_\infty$]
\label{thm:characterization_of_canonical_orientability}
The following are equivalent.
\begin{enumerate}
    \item 
 The extension $K_\infty$ is canonically orientable.

\item The representation
 $\Lie\Gamma(K_\infty)$ has no generalized Hodge-Tate weights in $\ZZ_{\leq-1}$. 

\end{enumerate}
\end{theorem}

\begin{proof}
Write $V = \Lie\Gamma(K_\infty)$.

Suppose item (1) holds. Then $\D_{\dR}^+(V(1)),$ which is the kernel of $\D_{\dR}^+(\Lie\Gamma(K_\infty)) \to \D_{\C}(\Lie\Gamma(K_\infty))$, must be zero. So the previous lemma implies item (2).

 Suppose item (2) holds. We have an exact sequence
\begin{equation} \label{eq:seq_for_ori_new}
0\rightarrow\D_{\dR}^{+}(V(1))\rightarrow\D_{\dR}^{+}(V)\rightarrow\D_{\C}(V)\rightarrow\mt{H}^{1}(G_{K},\B_{\dR}^{+}\otimes V(1)).
\end{equation} 
Since $V$ has no Hodge-Tate weights in $\ZZ_{\le -1}$, the twist $V(1)$ has no Hodge-Tate weights in $ \ZZ_{\le0}$. By the previous lemma applied to $V(1)$, $\D_{\dR}^+(V(1))$ and $\mt{H}^1(G_K,\B_{\dR}^+\otimes V(1))$ both vanish. Hence the map $\DdRplus(V)\to \D_{\C}(V)$ is an isomorphism and $K_\infty$ is canonically orientable.
\end{proof}

\subsection{Examples}
We now discuss specific examples of the concept of orientability.

\begin{example}[Abelian extensions are canonically orientable] \label{ex:ab_ext}
 Suppose $\Gamma(K_\infty)$ is abelian. The Hodge-Tate weights of $\Lie\Gamma(K_\infty)$
are $0$, so by Theorem \ref{thm:characterization_of_canonical_orientability} we get that $K_\infty$ is canonically orientable. We also know this independently of Theorem \ref{thm:characterization_of_canonical_orientability}, as we have already established in Corollary \ref{cor:sen_op_rationality} that the Sen operator lies in $K\otimes_{\QQ_p}\Lie\Gamma(K_\infty)$, and this clearly has a canonical lift to  $\DdRplus(\Lie\Gamma(K_\infty)).$

\end{example}

\begin{example}[Virtually nilpotent extensions are canonically orientable] \label{ex:nilp_ext}
Suppose more generally $\Gamma(K_\infty)$ is virtually nilpotent (i.e. a finite index subgroup of $\Gamma(K_\infty)$ is nilpotent). Then $ \Lie\Gamma(K_\infty)$ is nilpotent. The Sen operator of $\Lie\Gamma(K_\infty)$, as an element of $$\C \otimes_{\QQ_p} \Lie\Gamma(K_\infty) \subset \End_{\C}(\C \otimes_{\QQ_p} \Lie\Gamma(K_\infty)),$$ is equal to $\ad(\Theta(K_\infty))$. In particular, it is a  nilpotent operator, which means the generalized Hodge-Tate weights of $\Lie\Gamma(K_\infty)$ are all zero. It follows from Theorem \ref{thm:characterization_of_canonical_orientability} that $K_\infty$ is canonically orientable.
\end{example}

\begin{example}[The false Tate extension is canonically orientable] \label{ex:Kummer_ext} 
Let $\pi_0=\pi$ be a uniformizer of $K$, inductively for each $n \geq 1$, choose some $\pi_n \in \barK$ such that $\pi_n^p=\pi_{n-1}$. Let $\kinfty=\kft=\cup_n \kcyc(\pi_n)$ be the false Tate extension.  
 We claim that  $\Lie\Gamma(\kft)$ has Hodge-Tate weights
0 and 1, and so $K_\infty$ is canonically orientable by Theorem \ref{thm:characterization_of_canonical_orientability}. 
To see this, take the faithful representation $V$ of $\Gamma(\kft)$ given by
\[
g\mapsto\left[\begin{array}{cc}
\chi_{\cyc}(g) & c(g)\\
0 & 1
\end{array}\right],
\]
where $c(g)$ is the usual Kummer cocycle defined by the equation $g(\pi^{1/p^\infty}) = \zeta_{p^\infty}^{c(g)}\pi^{1/p^\infty}.$
Then we have
an identification
\[
\Lie\Gamma(\kft)=\left[\begin{array}{cc}
* & *\\
0 & 0
\end{array}\right]\subset\End_{\QQ_p}(V).
\]
Using the above description of $\Lie\Gamma(\kft)$, one computes it is the 2-dimensional
representation given by
\[
g\mapsto\left[\begin{array}{cc}
1 & 0\\
-c(g) & \chi_{\cyc}(g)
\end{array}\right],
\]
whose Hodge-Tate weights are $0,1$.  
Let us calculate the Sen operator and its lift in this case. We claim there is a unique element $\tilde{\alpha} \in \BdRplus(\kft)^\pa$ (respectively\footnote{This element appeared in \cite[{$\mathsection{4.5}$}]{berger2016theorie}.}  $\alpha \in \wh{K}_\mt{fT}^\la$) satisfying the identity $g(\tilde{\alpha})\chi_\cyc(g)=\tilde{\alpha}+c(g)$ (respectively $g(\alpha)\chi_\cyc(g)=\alpha+c(g)$). To show existence, take $\tilde{\alpha}=\log([\pi^\flat]/\pi)/t$, respectively $\theta(\tilde{\alpha})$. We explain uniqueness in the case of $\tilde{\alpha}$; the case of $\alpha$ is similar but easier. Suppose $\tilde{\alpha}_1, \tilde{\alpha}_2$ are two such elements, and let $z = \tilde{\alpha}_1-\tilde{\alpha}_2$. Then $g(\tilde{\alpha}_1-\tilde{\alpha}_2)\chi_{\cyc}(g)=\tilde{\alpha_1}-\tilde{\alpha}_2$. In particular, $z$ is fixed by $\Gal(\overline{K}/K_\cyc)$, hence is contained in $\BdRplus(K_\cyc)^\pa$, which is equal to $K_\cyc\llbracket t\rrbracket $ by \cite[Proposition 2.6(2)]{porat2022lubin}. Writing $z = \sum a_it^i$ with $a_i \in K_\cyc$, and observing the identity $\nabla(K_\cyc)(z) = -z$, we get $a_i = 0$ for all $i$, and so $z=0$, which shows $\tilde{\alpha}_1 = \tilde{\alpha}_2$. One then verifies that the Sen operator and its canonical lift are given by
\[
\Theta(\kft)=\left[\begin{array}{cc}
1 & \alpha\\
0 & 0
\end{array}\right] \text{and } \nabla(\kft)=\left[\begin{array}{cc}
1 & \tilde{\alpha}\\
0 & 0
\end{array}\right].
\]

In terms of more traditional notation (see \cite[Notation 6.5]{GMWHT}) this is \[\Theta(\kft) = \nabla_\gamma + \alpha \nabla_\tau \text{ and } \nabla(\kft) = \nabla_\gamma + \tilde{\alpha} \nabla_\tau.\]
\end{example}

\begin{example}[A non-orientable example] \label{ex_non_dR_ext}
 By \cite[Example 6.3.5]{brinon2009cmi}, we may choose a nonsplit extension
\[
0\rightarrow\QQ_{p}\rightarrow V\rightarrow\QQ_{p}(1)\rightarrow0.
\]
Let $K_{\infty}$ correspond to the kernel of this representation.
In matrix notation, we may write this representation as an extension
\[
g\mapsto\left[\begin{array}{cc}
1 & c(g) \\
0 & \chi_{\cyc}(g)
\end{array}\right]
\]
for some cocycle $c(g)$. As a subrepresentation of $\End(V)$, we
have 
\[
\Lie\Gamma(K_\infty)\cong\left[\begin{array}{cc}
0 & *\\
0 & *
\end{array}\right]\subset\left[\begin{array}{cc}
* & *\\
* & *
\end{array}\right]=\End(V).
\]
From this, it is straightforward to calculate $\Lie\Gamma(K_\infty)$ is the representation given by 
\[g\mapsto
\left[\begin{array}{cc}
\chi_{\cyc}(g)^{-1} & c(g)\chi_{\cyc}(g)^{-1}\\
0 & 1
\end{array}\right].
\]
In particular, it is Hodge-Tate with weights $-1,0$, but is not de
Rham (see \cite[Example 6.3.5]{brinon2009cmi}). As it is 2-dimensional, the map $\D_{\dR}^{+}(\Lie\Gamma(K_\infty))\rightarrow\D_{\C}(\Lie\Gamma(K_\infty))$ must be zero. So the element $\Theta(K_{\infty})\in\D_{\C}(\Lie\Gamma(K_\infty))$
 cannot be lifted to $\DdRplus(\Lie\Gamma(K_\infty)),$ which means $K_\infty$ is not orientable. We  remark that this non-orientability also follows from a more general criterion, see Theorem \ref{thm:characterization_of_orientability}. 
\end{example}

\begin{example}[${\mt{SL}_2}(\ZZ_p)$-examples] \label{ex:SL2}
Suppose that $\Gamma(K_\infty)\cong\mt{SL}_2(\ZZ_p)$. This case is discussed at length in \cite[{$\mathsection{5}$}]{berger2016theorie}. By considering the standard representation of $\mt{SL}_2(\ZZ_p)$, we get a representation of $G_K$ with two generalized Hodge-Tate weights $\pm s$ for some $s \neq 0$  in $\bar{\QQ}_p$. The representation $\Lie\Gamma(K_\infty)$ then has generalized Hodge-Tate weights $0, \pm 2s$.

There are two cases:
\begin{enumerate}
\item   $s\notin\frac{1}{2}(\ZZ\backslash\{0\})$. In this case,
we have $\pm2s\notin\ZZ_{\leq-1}$, so Theorem \ref{thm:characterization_of_canonical_orientability} applies, and $K_\infty$ is canonically orientable.

\item  $s\in\frac{1}{2}(\ZZ\backslash\{0\})$. In this case, either $K_\infty$ is non-orientable, or 
$K_\infty$ is orientable, but necessarily not canonically so, because $\D_{\dR}^{+}(\Lie\Gamma(K_\infty)(1))$
is nonzero. This can happen, for example, if one takes a modular
form $f$ of weight 2, which is non-CM and non-ordinary at $p$, and
twists the restrictions to $G_{\QQ_p(\mu_{2p})}$ of the Galois representation of $V_{f}$ by
$\det(V_{f})^{-1/2}$. Later in Corollary  \ref{cor:orientable_iff_HT_eq_de_rham}, we will see that for $\Lie\Gamma(K_\infty)$ Hodge-Tate, being orientable is equivalent to being de Rham.

\end{enumerate}
\end{example}

\subsection{Compatibilities}

In the next section we shall need functoriality properties of orientations. The following result shows that orientations act compatibly along towers; this can be regarded as a compatibility complementing Lemma \ref{lemma:sen_operator_compatibility}.

\begin{lemma} 
\label{lem:nabla_action_compatibility}
Suppose that $K_\infty \subset L_\infty$ are $p$-adic Lie extensions of $K$ with compatible orientations $\nabla(L_\infty)$, $\nabla(K_\infty)$, i.e. $\nabla(L_\infty)$ is mapped to $\nabla(K_\infty)$ under the map $V_{\dR}^+(L_\infty) \to V_{\dR}^+(K_\infty)$. Let $x \in \BdRpluspa$. Then $\nabla(L_\infty)(x) = \nabla(K_\infty)(x).$
\end{lemma}

\begin{proof} Let $\pi:\Gamma(L_\infty) \to \Gamma(K_\infty)$ be the projection with the induced projection $\Lie\pi:\LieGamma(L_\infty)\to \Lie\Gamma(K_\infty)$. The element $x$ lies in the subring $\BdRplus(K_\infty)^{\pa}$, hence if $\sum b_i \otimes \partial_i \in \BdRplus(L_\infty)^{\pa} \otimes_{\QQ_p} \LieGamma(L_\infty)$ then $(\sum b_i \otimes \partial_i)(x)=(\sum b_i \otimes (\Lie\pi)(\partial_i))(x).$ On the other hand, $(1\otimes \Lie \pi)$ induces the natural map from $V_{\dR}^+(L_\infty)\to V_{\dR}^+(K_\infty)$, so by assumption $(1\otimes \Lie\pi)(\nabla(L_\infty)) = \nabla(K_\infty)$. Thus we get
$$\nabla(L_\infty)(x) = (1\otimes \Lie\pi)(\nabla(L_\infty))(x) = \nabla(K_\infty)(x),$$as required. \end{proof}

We will conclude the section by proving a result that will help us shift orientations along central extensions.

\begin{lemma}
\label{lem:orientations_preserved_under_central_extensions}
Let $L_\infty/K$ be p-adic Lie Galois extension with $K_\infty \subset L_\infty$ and $\Delta = \Gal(L_\infty/K_\infty)$ abelian. Suppose $\Delta$ is central in $\Gamma(L_\infty)=\Gal(L_\infty/K)$ up to a finite extension of $K$; in other words, that the conjugation map $\Gamma(K_\infty) = \Gal(K_\infty/K)\to \mt{Aut}(\Delta)$ has finite image. Then the natural map $V_{\dR}^+(L_\infty)\to V_{\dR}^+(K_\infty)$ is an isomorphism and induces a bijection between orientations of $L_\infty$ and orientations of $K_\infty$.
\end{lemma}

\begin{proof}
To make the notation lighter, we write for this proof only $\DdRplus(K_\infty)$ for $\DdRplus(\Lie\Gamma(K_\infty))$; similarly with $\D_C$ instead of $\DdRplus$ and $L_\infty$ instead of $K_\infty$.
We are free to replace $K$ by a finite extension contained in $K_\infty$, and so we may assume $\Delta$ is central in $\Gamma(L_\infty)$. We have an exact sequence
$$1 \to \Delta \to \Gamma(L_\infty) \to \Gamma(K_\infty) \to 1$$
of $p$-adic Lie groups. We may consider $\Lie(\Gamma(L_\infty))$ as a $G_K$-representation by letting $G_K$ first project onto $\Gamma(L_\infty)$ and then act via the adjoint action. We have an exact sequence of $G_K$-representations with $\QQ_p$-coefficients:
$$0 \to \Lie\Delta \to \Lie\Gamma(L_\infty) \to \Lie\Gamma(K_\infty) \to 0,$$
and $\Lie\Delta$ is the trivial representation, since $\Delta$ is central.

We now argue in steps:

\textbf{Step 1}: Consider the exact sequence obtained by tensoring with $\C$ and taking $G_K$-fixed points:
$$ 0 \to \D_{\C}(\Lie\Delta) \to \D_{\C}(L_\infty) \to \D_{\C}(K_\infty) \to \mt{H}^1_\C(\Lie\Delta) \to ...$$

Then we claim that $V_{\Sen}(K_\infty)$ is in the kernel of the connecting map from $\D_{\C}(K_\infty)$ to $\mt{H}^1_\C(\Lie\Delta)$. Indeed, by Lemma \ref{lemma:sen_operator_compatibility}, $\Theta(K_\infty)$ is the image of $\Theta(L_\infty)$ under the map from $\D_{\C}(L_\infty)$ to $ \D_{\C}(K_\infty)$, hence it is mapped to zero in the complex.

\textbf{Step 2}: Let $\partial$ be the boundary map $$\partial:\D_{\C}(K_\infty) \to \mt{H}^1_\C(\Lie\Delta).$$ Consider the following commutative diagram:
\[\begin{tikzcd}[row sep=small, column sep=small] 
    0 \arrow[r] 
    & \D_{\C}(\Lie\Delta) \arrow[d,"g"] \arrow[r] 
        & \D_{\C}(L_\infty) \arrow[d] \arrow[r] 
        & \ker(\partial) \arrow[d] \arrow[r] 
        & 0 \\
    0 \arrow[r] 
    & \ker(f)  \arrow[r] 
        & \D_{\C}(L_\infty)/V_{\Sen}(L_\infty) \arrow[r,"f"] 
        & \D_{\C}(K_\infty)/V_{\Sen}(K_\infty)
\end{tikzcd}
\]

Both rows of the diagram are exact. We claim that the map $g$ is an isomorphism, so that $\ker(f)$ is identified with $\D_{\C}(\Lie\Delta)$. Indeed, we have by the snake lemma an exact sequence

$$ 0 \to \ker(g) \to V_{\Sen}(L_\infty) \to V_{\Sen}(K_\infty) \to \coker(g) \to 0,$$
 where we use the previous step to identify the third term with $V_{\Sen}(K_\infty)$. Since $V_{\Sen}(L_\infty) \to V_{\Sen}(K_\infty)$ is an isomorphism, we conclude the proof of the second step.

\textbf{Step 3:} Using the previous step, we get another commutative diagram with exact rows
\[
\begin{tikzcd}[row sep=small, column sep=small] 
    0 \arrow[r] 
    & \DdRplus(\Lie\Delta) \arrow[d] \arrow[r] 
        & \DdRplus(L_\infty) \arrow[d] \arrow[r] 
        & \DdRplus(K_\infty) \arrow[d] \arrow[r] & \mt{H}^{1,+}_{\dR}(\Lie\Delta) \arrow[d] \\
    0 \arrow[r] 
    & \D_{\C}(\Lie\Delta)  \arrow[r] 
        & \D_{\C}(L_\infty)/V_{\Sen}(L_\infty) \arrow[r] 
        & \D_{\C}(K_\infty)/V_{\Sen}(K_\infty) \arrow[r] 
        & \mt{H}^1_\C(\Lie\Delta) 
\end{tikzcd}\]

The two outer vertical maps are isomorphisms, because $\Lie\Delta$ is a trivial $G_K$-representation. By a diagram chase, we conclude that the natural map $V_{\dR}^+(L_\infty) \to V_{\dR}^+(K_\infty)$, which is the map between the kernels of the middle columns of the diagram, is an isomorphism. The claim follows.\end{proof}

\begin{corollary}
\label{cor:bijection_between_orientaitons_along_dim_1_extensions}
Let $L/K$ be an abelian 1-dimensional $p$-adic Lie extension and let $L_\infty=K_\infty L$. Then $L_\infty/K$ is a $p$-adic Lie extension, the natural map $V_{\dR}^+(L_\infty)\to V_{\dR}^+(K_\infty)$ is an isomorphism, and there is a bijection between orientations of $L_\infty$ and orientations of $K_\infty$.\end{corollary}

\begin{proof}
The injection $\Gamma(L_\infty) \to \Gamma(K_\infty/K) \times \Gal (L/K)$ identifies $\Gamma(L_\infty)$ with a closed subgroup of the target, so it is a $p$-adic Lie group. Furthermore, under this injection, the subgroup $\Delta=\Gal(L_\infty/K_\infty)$ is identified with a subgroup of $\Gal(L/K)$. In particular, it follows that the relevant conjugation action on $\Delta$ has finite image. Now apply the previous lemma.\end{proof}

\section{Pro-analytic vectors in  {$\BdRplus$}}
\label{section:la_vecs_in_bdr}

In this section we study the structure of $\BdRpluspa$. The main purpose of this section is to prove Theorem \ref{thm:brdplusps_iso}, which says that an orientation gives rise to a Galois-equivariant description of $\BdRpluspa$ in terms of a 1-variable power series ring over $\Kinftyla$. In \S \ref{subsec:kcyc_case}, we first treat the case $K_\cyc \subset K_\infty$; we then treat the general case in \S \ref{subsec:gen_case}.
This canonical description makes it possible to relate $\bdrplus$-representations to \emph{regular connections} in \S \ref{sec:reg_conn}.
Before doing that, we first prove in \S \ref{subsec:alg_bdr_plus_iso} that even without an orientation, $\BdRpluspa$ is abstractly isomorphic to such a ring.

\subsection{The algebraic structure of $\bdrplus(\kinfty)^\pa$}
\label{subsec:alg_bdr_plus_iso}

Recall our notation from \S \ref{subsubsection:p_adic_hodge_theory_notations}. The goal of this subsection is to prove that $\bdrplus(\kinfty)^\pa$ and $\Kinftyla \llbracket s \rrbracket$ are abstractly isomorphic as rings (Theorem \ref{thm:dvr}).

We shall first need the following lemma.

\begin{lemma}
\label{lem:graded_analytic_quotients}
For $n\ge 0$, the natural map $(I_\theta^n)^\pa/(I_\theta^{n+1})^{\pa} \to (I_\theta^n/I_\theta^{n+1})^\la$ is an isomorphism.
\end{lemma}

\begin{proof}
The map is clearly injective. For each $m \geq 0$ we have $\R^1_\la(I_\theta^m/I_\theta^{m+1})=0$ by Theorem \ref{thm:higher_la_vanishing_k_infty}. It follows that the map  $(I_\theta^n/I_\theta^{m+1})^\la \to (I_\theta^n/I_\theta^{m})^\la$ is surjective for all $m\ge n$. This implies the surjectivity of $(I_\theta^n)^\pa \to (I_\theta^n/I_\theta^{n+1})^\la$.
\end{proof}

\begin{proposition}
\label{prop:bdran_local}

The ring $\bdrplus(\kinfty)^\pa$ is a complete discrete valuation ring with maximal ideal $I_{\theta}^{\pa}$ and residue field $\Kinftyla$.
\end{proposition}

\begin{proof}
We know that $\hatkinftyla$ is   a field by Lemma \ref{lem:basis_lav}. By the previous lemma, we have $\BdRpluspa/I_\theta^\pa \cong \Kinftyla$, which shows that $I_\theta^\pa$ is maximal with residue field $\Kinftyla$. By Lemma \ref{lem:basis_lav} we have 
$$\B_{\dR}^{+}(K_\infty)^{\pa,\times}=\B_{\dR}^{+}(K_\infty)^{\pa}\cap\B_{\dR}^{+}(K_\infty)^{\times},$$
so any element not in $I_\theta^\pa$ is invertible. Therefore $\B_{\dR}^{+}(K_\infty)^{\pa}$ is local.

Let $n\ge 1.$ Since $I_\theta^n/I_\theta^{n+1}$ is 1-dimensional over $\hatkinfty$, its locally analytic vectors form   a one-dimensional vector space over $\hatkinftyla$ (\cite[Theorem 3.4]{berger2016theorie}). By the previous lemma, we may choose $s \in I_\theta^\pa\setminus (I_\theta^2)^\pa$. Then $s$ generates $I_\theta$ and thus $s^n \in (I_\theta^n)^\pa \setminus (I_\theta^{n+1})^\pa$.  Lemma \ref{lem:basis_lav} then implies that  
 $(I_\theta^n)^\pa$ is the ideal in $\BdRpluspa$ generated by $s^n$. In particular, $I_\theta^\pa$ is principal, so Exercise 4 of \cite[Chapter 9]{atiyah2018introduction} implies $\BdRpluspa$ is a discrete valuation ring. It remains to prove $\BdRpluspa$ is $(I_\theta^\pa)^n = (I_\theta^n)^\pa$-adically complete. This comes down to showing $\BdRpluspa/(I_\theta^n)^\pa = (\BdRplus(K_\infty)/I_\theta^n)^\la$, which follows by a d\'evissage argument from the previous lemma.
\end{proof}

Cohen's structure theorem (\cite[Tag 0C0S]{Sta}) then implies the following result.

\begin{theorem}
    \label{thm:dvr}
         There is an algebraic  isomorphism
\[\BdRpluspa \cong \widehat{K}_{\infty}^{\la}\llbracket s \rrbracket \]
which maps $s$ to a generator of $I_{\theta}^{\pa}$.
 \end{theorem}

\begin{corollary}
   The map $\B_{\dR}^{+}(K_\infty)^{\pa}\rightarrow\B_{\dR}^{+}(K_\infty)$
is faithfully flat. 
\end{corollary}

\subsection{Statement of the main theorem}

In the presence of an orientation we can get a much stronger $\Gamma(K_\infty)$-equivariant version of Theorem \ref{thm:dvr}.

\begin{theorem}
\label{thm:brdplusps_iso}
Let $K_\infty/K$ be an orientable $p$-adic Lie Galois extension with a chosen orientation $\nabla=\nabla(K_\infty)$. Then:

\begin{enumerate}
    \item  The composition $$\BdRplus(K_\infty)^{\pa,\nabla=0} \subset \BdRpluspa \twoheadrightarrow \Kinftyla$$ induces a 
 natural $\Gamma(K_\infty)$-equivariant isomorphism of $K$-algebras
$$\BdRplus(K_\infty)^{\pa,\nabla=0} \cong \Kinftyla.$$ In particular, $\BdRpluspa$ acquires in this way a natural $\Kinftyla$-algebra structure, and the structure morphism $\Kinftyla \to \BdRpluspa$ is $\Gamma(K_\infty)$-equivariant. 

\item  For $n \ge 1$, the $\Gamma(K_\infty)$-stable $\Kinftyla$-subspace $$\BdRplus(K_\infty)^{\pa,\nabla=n} \subset\BdRplus(K_\infty)$$ is contained in $\Fil^n\BdRpluspa$. The induced map $$\BdRplus(K_\infty)^{\pa,\nabla=n} \to L(K_\infty)^{\otimes n,\la}=\mt{gr}^n\BdRplus(K_\infty)^\pa$$ is a natural isomorphism of 1-dimensional $\Kinftyla$-semilinear representations of $\Gamma(K_\infty)$.

\item The natural map from $\prod_{n \ge 0} \BdRplus(K_\infty)^{\pa,\nabla=n}$ to $\BdRpluspa$ is an isomorphism, and so one gets a natural
$\Gamma(K_\infty)$-equivariant isomorphism of $\wh{K}_\infty^\la$-algebras
$$\BdRpluspa \cong \wh{\Sym}_{\Kinftyla}(L(K_\infty)^\la).$$ 
\end{enumerate}
\end{theorem}

\begin{remark}
In particular, choosing a generator $t_{K_\infty}$ of $L(K_\infty)^\la$, we get a $\Gamma(K_\infty)$-trivialization $L(K_\infty)^\la\cong \Kinftyla\cdot t_{K_\infty}.$
By considering $t_{K_\infty}$ as an element in $\mathrm{Fil}^1 \bdrpluskinftypa$, we get that the natural inclusion map
\[ \BdRplus(K_\infty)^{\pa,\nabla=0}\llbracket t_\kinfty \rrbracket \to  \bdrpluskinftypa\]
is an isomorphism. This induces the $\Gamma(K_\infty)$-equivariant isomorphism
\[  \wh{K}_\infty^\la\llbracket t_{K_\infty}\rrbracket \cong \BdRpluspa\]
in Question \ref{question:bdr_plus_iso}.
\end{remark}

\begin{example}[The False Tate extension]

Let $\kft, \tilde{\alpha}$ and $\alpha$ be as in Example \ref{ex:Kummer_ext}. 
  We describe explicitly the \emph{unique} Galois-equivariant section map
  \[ s: \kftla \to \bdrplus(\kft)^\pa.\]
  We shall see it is completely determined by one rule $s(\alpha)=\wt\alpha$. By \cite[Proposition 4.12]{berger2016theorie}, we have
      \[ \kftla =\cup_{n \geq 1} K({\zeta_{p^{r(n)}}, \pi_{r(n)}})\{\{ \alpha-\alpha_n \}\}_n. \]
      See~ \cite[\S 4]{berger2016theorie}   for the notation. We just mention here that the $\alpha_n \in \kft$ is a sequence of (algebraic) elements approximating $\alpha$, and the expression above says that each $f \in \wh{K}_{\mathrm{fT}}^\la$ can be expressed as a (converging) power series with coefficients in $\kft$ and with variable $(\alpha-\alpha_n)$ for some $n$.
It suffices to determine $s(\alpha)$, as everything else in the expression of $\kftla$ are elements in $\kft$ which embed  canonically into $\bdrplus$. In fact, we claim $s(\alpha)=\wt\alpha$.
  Indeed, as the section $s$ comes from the isomorphism 
  $\bdrplus(\kft)^{\pa, \nabla(\kft) =0}\cong   \kftla$,  we reduce to verifying  
  $\nabla(\wt\alpha)=0,$
  which follows from the computation of Example \ref{ex:Kummer_ext}.
\end{example}

\subsection{The case {$K_\cyc \subset K_\infty$}} \label{subsec:kcyc_case}

\begin{lemma} \label{lem:nabla_on_t}
Suppose that $K_\cyc \subset K_\infty$ with $K_\infty$ orientable. Fix an orientation $\nabla(K_\infty)$, so that $\nabla(K_\infty)$ acts on $\BdRplus(K_\infty)^{\pa}$. Then $\nabla(K_\infty)(t) = t$, where $t$ is Fontaine's element. In particular, $\nabla(K_\infty)((I_\theta^\pa)^n) \subset (I_\theta^\pa)^n$ for $n \ge 0$.
\end{lemma}

\begin{proof} Since $K_\cyc$ is canonically orientable, the operator $\nabla(K_\infty)$ must map to $\nabla(K_\cyc)$ under the map $V_{\dR}^+(K_\infty) \to V_{\dR}^+(K_\cyc)$. Now use Lemma \ref{lem:nabla_action_compatibility} and the relation $\nabla(K_\cyc)(t) = t$. The last statement of the lemma follows from the Leibniz rule.\end{proof}

\begin{lemma}  
\label{lem:nabla_action_on_graded_pieces}
  For $n \ge 0$, the operator $\nabla(K_\infty)$ acts on $(I_\theta^\pa)^n/(I_\theta^\pa)^{n+1}$ multiplication by $n$.
\end{lemma}
\begin{proof}
    As $\kcyc \subset \kinfty$, we can take $s=t$ in proof of Proposition \ref{prop:bdran_local}. Recall that by Lemma \ref{lem:graded_analytic_quotients} and Proposition \ref{prop:bdran_local}, we have $(I_\theta^\pa)^n/(I_\theta^\pa)^{n+1}\cong [I_\theta^n/I_\theta^{n+1}]^\la$. Then $[I_\theta^n/I_\theta^{n+1}]^\la\cong\hatkinftyla\cdot t^n$ and we conclude by the previous lemma and Proposition \ref{prop:sen_operator_kills_la_vecs}.
\end{proof}

We shall need the following key lemma.

\begin{lemma}
\label{lem:operator_on_filtered_space}
Let $A$ be a $\QQ$-algebra, complete and separated with respect to a decreasing filtration $A = \Fil^0A \supset \Fil^1A \supset ...$, with $\gr^0A$ a field and each $\gr^iA$ a finite-dimensional vector space over $\gr^0A$. Let $T$ be a $\QQ$-linear operator on $A$ which respects the filtration and acts on $\gr^iA$ as multiplication by $i$.  Then:

\begin{enumerate}
    \item 
 For all $i \ge 0$ we have $A^{T=i} \subset \Fil^iA$ and the induced map $A^{T=i} \to \gr^iA$ is an isomorphism.

\item The natural map $\prod_{i \ge 0} A^{T=i} \to A$ is an isomorphism. 

\end{enumerate}
\end{lemma}

\begin{proof}
We start by proving $A^{T=i} \subset \Fil^iA$. For $ 0 \le j <i$, we have an exact sequence of $\QQ[T]$-modules
\begin{align*}
0 \to \Fil^{j+1}A \to \Fil^jA \to \gr^jA \to 0.
\end{align*}
Since $\Hom_{\QQ[T]}(A^{T=i},\gr^jA) = 0,$ we get $\Hom_{\QQ[T]}(A^{T=i},\Fil^jA) = \Hom_{\QQ[T]}(A^{T=i},\Fil^{j+1}A).$
Applying this consecutively for $j = 0,...,i-1$, starting from the inclusion $A^{T=i} \subset  \Fil^0A$, we get $A^{T=i} \subset \Fil^iA.$ 

We now show $A^{T=i} \to \gr^iA$ is an isomorphism. Since $A = \varprojlim_n A/\Fil^nA$, it suffices to prove that $[A/\Fil^nA]^{T=i} \to \gr^iA$ is an isomorphism for $n \ge i+1$. The assumption that $T$ acts as multiplication by $j$ on $\gr^j$ implies by d\'evissage that $$[\Fil^iA/\Fil^nA]^{T=i}=[A/\Fil^nA]^{T=i},$$
so we reduce to proving $[\Fil^iA/\Fil^nA]^{T=i} \to \gr^iA$ is an isomorphism. Argue by induction on $n \ge i+1$, the base case being obvious. Then for the induction step, it suffices to show $[\Fil^iA/\Fil^{n+1}A]^{T=i} \to [\Fil^iA/\Fil^{n}A]^{T=i}$ is an isomorphism. To see this, apply the functor $\Hom_{\QQ[T]}(\QQ[T]/(T-i),*)$ to the sequence
\begin{align*}
0 \to \gr^nA \to \Fil^iA/\Fil^{n+1}A \to \Fil^iA/\Fil^{n}A \to 0
\end{align*} and observe that $T-i$ acts as an isomorphism on $\gr^nA$. This completes the proof of item (1).

For item (2), we write  $$A_m:=(\prod_{i\ge 0}A^{T=i})/(\prod_{i\ge m+1}A^{T=i}).$$ It suffices to show that for each $m$, the map $A_m \to A/\Fil^{m+1}A$ is an isomorphism. We prove this claim by induction on $m$. For $m=0$, the desired claim is the same as the isomorphism $A^{T=0} \cong \gr^0A,$ which we have already shown. Then by induction we consider the commutative diagram

\[
\begin{tikzcd}[row sep=small, column sep=small]
    0 \arrow[r] 
    & A^{T=m+1}  \arrow[d] \arrow[r] 
        & A_{m+1} \arrow[d] \arrow[r] 
        & A_m \arrow[d] \\
    0 \arrow[r] 
    & \gr^{m+1}A  \arrow[r] 
        & A/\Fil^{m+2}A\arrow[r] 
        & A/\Fil^{m+1}A.
\end{tikzcd}
\]
We know that the left vertical arrow is an isomorphism; the right vertical arrow is an isomorphism by induction. It follows the middle vertical map is injective, but it is also a map of $\gr^0A$-vector spaces of the same finite dimension, hence it must be an isomorphism.
\end{proof}

The following proposition completes the proof of Theorem \ref{thm:brdplusps_iso} in the case where $K_\cyc \subset K_\infty$. In this case, there is a trivialization $L(K_\infty)^\la\cong \Kinftyla\cdot t$, and so the statement simplifies.

\begin{proposition}
\label{prop:bdrpluspa_iso_power_series_ring}
Suppose that $K_\cyc \subset K_\infty$ with $K_\infty$ orientable. Fix an orientation $\nabla(K_\infty)$. Then, depending on the orientation $\nabla(K_\infty)$, the ring $\B_{\dR}^{+}(K_\infty)^{\pa}$ is naturally a $\wh{K}_{\infty}^{\la}$-algebra, and there exists a natural isomorphism of $\wh{K}_{\infty}^{\la}$-algebras
$$\B_{\dR}^{+}(K_\infty)^{\pa} \cong   \wh{K}_{\infty}^{\la}\llbracket t\rrbracket ,$$
where $t$ is Fontaine's element. In this isomorphism, $\B_{\dR}^{+}(K_\infty)^{\pa,\nabla=n}$ is identified with $\Kinftyla\cdot t^n.$  
\end{proposition}

\begin{proof}
This follows from Lemma \ref{lem:operator_on_filtered_space} by taking $A = \BdRpluspa$ with its $I_\theta^\pa$-adic filtration and $T=\nabla$. To see the lemma applies, use Proposition \ref{prop:bdran_local}, Lemma \ref{lem:nabla_on_t} and Lemma \ref{lem:nabla_action_on_graded_pieces}.
\end{proof}



\subsection{The general case} \label{subsec:gen_case}

In this subsection, we complete the proof of Theorem \ref{thm:brdplusps_iso}. We let $K_\infty$ be general in this section.

\begin{lemma}
\label{lem:reduction_along_finite_extensions_easy_direction}
If Theorem \ref{thm:brdplusps_iso} is true for $K$, and $K'/K$ is a finite  extension, then it is also true for $K'$.
\end{lemma}
\begin{proof}
As usual let $K_\infty' = K_\infty K'$. If $K_\infty'=K_\infty$ the claim is clear. By the same argument as that appearing in the proof of Lemma \ref{lem:reduction_finite_extension_vanishing_higher_la}, we may reduce to the case where $K'$ is linearly disjoint from $K_\infty$. In this case we have a natural identification $\Gamma(K_\infty')\cong \Gamma(K_\infty)$ as well as natural $\Gamma(K_\infty')$-equivariant isomorphisms $K_\infty' \cong K_\infty\otimes_K K'$, $L(K_\infty') \cong L(K_\infty)\otimes_K K'$ and $\BdRplus(K_\infty') \cong \BdRplus(K_\infty)\otimes_K K'$. The result follows from base change by Lemma \ref{lem:basis_lav}.
\end{proof}

\begin{lemma}
\label{lem:reduction_along_finite_extensions}
If Theorem \ref{thm:brdplusps_iso} is true after replacing $K$ by a finite  extension $K'/K$, then it is also true for $K$.
\end{lemma}

\begin{proof}
By the previous lemma, we may assume $K'$ is Galois over $K$. Set $K'_\infty=K_\infty K'$ once more and let $\Delta=\Gal(K_\infty'/K_\infty)$ be the finite Galois group of $K_\infty'$ over $K_\infty$. Let $\tilde{\Gamma}(K_\infty') = \Gal(K_\infty'/K)$. It is an extension of $\Gamma(K_\infty')$ by $\Delta$. 

Let us first choose an orientation $\nabla(K_\infty')$ for $K_\infty'$. Observing that the natural map 
$\Lie\Gamma(K_\infty') \to \Lie\Gamma(K_\infty)$ is an isomorphism of $\QQ_p$-vector spaces, we may and do consider $\Lie\Gamma(K_\infty')$ as a $\tilde{\Gamma}(K_\infty')$-representation (rather than just a $\Gamma(K_\infty')$-representation). We then have an identification $$V_{\dR}^+(K_\infty')^{\Gal(K'/K)}= V_{\dR}^+(K_\infty).$$ We choose the orientation $\nabla(K_\infty')$ corresponding to $\nabla(K_\infty)$ under this isomorphism. It has the property of being $\tilde{\Gamma}(K_\infty')$-equivariant (rather than just being $\Gamma(K_\infty')$-equivariant).

The statements of Theorem \ref{thm:brdplusps_iso} for $K_\infty$ are now easily checked once one makes the easy observation that each of the maps for $K_\infty'$ is actually $\tilde{\Gamma}(K_\infty')$-equivariant. For example, for item (1) of Theorem \ref{thm:brdplusps_iso}, we have by assumption that the canonical map from $\BdRplus(K_\infty')^{\pa,\nabla(K_\infty')=0}$ to $\Ktaginftyla$ is a $\Gamma(K_\infty')$-equivariant isomorphism. But by construction this map is the composition of $\BdRplus(K_\infty')^{\pa,\nabla(K_\infty')=0} \subset \BdRplus(K_\infty')^{\pa}$ and $\BdRplus(K_\infty')^{\pa} \to \Ktaginftyla$, both of which are $\tilde{\Gamma}(K_\infty')$-equivariant. Taking $\Delta$-fixed points proves the natural map from $\BdRplus(K_\infty)^{\pa,\nabla(K_\infty)=0}$ to $\Kinftyla$ is an isomorphism. The proofs of item (2) and item (3) are similar.
\end{proof}

For the rest of this subsection, let $L_\infty = K_\infty K_\cyc$ and $\Delta=\Gal(L_\infty/K_\infty)$. If $L_\infty/K_\infty$ is a finite extension, then by Lemma \ref{lem:reduction_along_finite_extensions}, Theorem \ref{thm:brdplusps_iso} holds. Thus as in Section \ref{section:vanishing_of_higher_la}, we may put ourselves in the situation where $L_\infty/K_\infty$ is 1-dimensional extension, $\Delta\cong\ZZ_p$ with generator $\gamma$, and there is an element $z \in \wh{L}_\infty^{\Gamma(L_\infty)\hla}$ satisfying $\gamma(z) = z+1.$ Furthermore, because of Corollary \ref{cor:bijection_between_orientaitons_along_dim_1_extensions}, the orientation $\nabla(K_\infty)$ of $K_\infty$ induces a \emph{unique} orientation $\nabla(L_\infty)$ of $L_\infty$ which is compatible with the inclusion $K_\infty \subset L_\infty$.


The following key construction defines a ``$\kinfty$-analogue" of Fontaine's  element $t$.

\begin{construction} \label{cons:t_kinfty}
Take $z_0 \in L_\infty$ close enough to $z$ so that the series $$u = \exp(-\log\chi_{\cyc}(\gamma)(z-z_0))$$ converges inside $\wh{L}_\infty^\la$, which, via Proposition \ref{prop:bdrpluspa_iso_power_series_ring}, can be  regarded as an element in $\bdrplus(L_\infty)^{\pa, \nabla(L_\infty)=0}$.
Then after possibly replacing $K$ by a finite extension, we get that $u$ has the property that $g(u) = \chi_{\cyc}(g)^{-1}u$ for $g \in \Delta$. Thus $t_{K_\infty}=ut$ is fixed by $\Delta$, and this means that
\[  t_{K_\infty} \in \mathrm{Fil}^1 \BdRpluspa.\]
Note $\nabla(L_\infty)( t_{K_\infty})= t_{K_\infty}$ by Lemma \ref{lem:nabla_on_t}. Thus, by  Lemma \ref{lem:nabla_action_compatibility}, we have
\[\nabla(K_\infty)( t_{K_\infty})= t_{K_\infty}.\]
\end{construction}

The construction implies:

\begin{lemma}
We have $\nabla(K_\infty)((I_\theta^\pa)^n)\subset (I_\theta^\pa)^n$ and $\nabla(K_\infty)$ acts on $(I_\theta^\pa)^n/(I_\theta^\pa)^{n+1}$ as multiplication by $n$.
\end{lemma}

\begin{proof}
The first claim holds because it is true for $L_\infty$, see Lemma \ref{lem:nabla_action_compatibility}. For the second claim argue as in the proof of Lemma \ref{lem:nabla_action_on_graded_pieces}, using $t_{K_\infty}$ instead of $t$.
\end{proof}



The following Proposition then concludes the proof of Theorem \ref{thm:brdplusps_iso}.
\begin{proposition} 
The map
$$\wh{K}_\infty^\la\llbracket t_{K_\infty}\rrbracket  \to \BdRpluspa$$ is a $\Gamma(K_\infty)$-equivariant isomorphism.
\end{proposition}

\begin{proof}
This follows from the previous lemma and Lemma \ref{lem:operator_on_filtered_space}.
\end{proof}



\section{Consequences for orientations}

In the previous section we studied the structure of $\BdRplus(K_\infty)^{\pa}$, notably via Theorem \ref{thm:brdplusps_iso}. In this section, we show how the structural results there clarify further the concept of an orientation.
In Theorem \ref{thm:characterization_of_orientability}, we give an equivalent criterion of orientability (compare with Theorem \ref{thm:characterization_of_canonical_orientability} on criterion of canonical orientability). 
Then in Theorem \ref{thm:parameterization_of_bdrplus_isos}, we classify equivariant power series descriptions of $\BdRpluspa$,  giving   a complete answer to Question \ref{question:bdr_plus_iso} from the introduction.

\subsection{Representations of $\Gamma(K_\infty)$ and a criterion for orientability}

We make the following observation: even though the condition of orientability asks for the lifting of a single element $\Theta(K_\infty)$ in $\D_{\C}(\Lie\Gamma(K_\infty))$ to $\DdRplus(\Lie\Gamma(K_\infty))$, one automatically gets from this a splitting of the entire space $\D_{\C}(V)$ for an arbitrary $p$-adic representation $V$ of $\Gamma(K_\infty)$. In what follows, recall that for such $V$ we have identifications $\D_{\C}(V) = (\wh{K}_\infty^\la\otimes_{\QQ_p} V)^{\Gamma(K_\infty)}$ and $\DdRplus(V) = (\BdRpluspa\otimes_{\QQ_p} V)^{\Gamma(K_\infty)}$.

\begin{corollary} \label{cor:orientable_iff_HT_eq_de_rham}
Let $K_\infty$ be orientable. Let $V$ be a representation of $\Gamma(K_\infty)$, considered as a $G_K$-representation. 
\begin{enumerate}
    \item The natural map $\DdRplus(V)\to\D_{\C}(V)$ has a splitting. In particular, $\gr^0\D_{\dR}(V)\cong \D_{\C}(V).$

    \item We have $\dim_K\D_{\mt{HT}}(V) = \dim_K\D_{\dR}(V)$. In particular,
$V$ is Hodge-Tate if and only if it is de Rham.
\end{enumerate} 
\end{corollary}
\begin{proof}
For item (1), use the existence of a $G_K$-equivariant section for the map $\BdRpluspa\to \wh{K}_\infty^\la$ (Theorem \ref{thm:brdplusps_iso}).

For item (2), we may regard $V$ as a representation of $\Gamma(\kinfty \kcyc)$ via its surjection to $\Gamma(\kinfty)$. Corollary \ref{cor:bijection_between_orientaitons_along_dim_1_extensions}  implies $\kinfty \kcyc$ is still orientable. Thus, without loss of generality, we may assume $K_\infty$ contains $K_\cyc$, so that twists $V(i)$ of $V$ are still representations of $\Gamma(K_\infty)$.
Item (1)   then implies that $$\gr^i\D_{\dR}(V)\cong \gr^0\D_{\dR}(V(i)) = \D_{\C}(V(i)),$$ hence we can sum over $i \in \ZZ$ to conclude.
\end{proof}



Furthermore, we get the following simple characterization of orientable $K_\infty$:

\begin{theorem}
\label{thm:characterization_of_orientability}
The following are equivalent.

\begin{enumerate}
\item The extension $K_\infty$ is orientable.

\item $\dim_K\D_{\dR}(\Lie\Gamma(K_\infty)) = \dim_K\D_{\mt{HT}}(\Lie\Gamma(K_\infty)).$
\end{enumerate}
\end{theorem}

\begin{proof}
Corollary  \ref{cor:orientable_iff_HT_eq_de_rham} shows that item (1) implies item (2). Conversely, suppose item (2) holds. We have
\begin{align*}
    \dim_K\D_{\dR}(\Lie\Gamma(K_\infty)) &= \sum_{i \in \ZZ} \dim_K \gr^i \D_{\dR}(\Lie\Gamma(K_\infty))\\
    &\le \sum_{i \in \ZZ} \dim_K \D_{\C}(\Lie\Gamma(K_\infty)(i))\\
    &= \dim_K\D_{\mt{HT}}(\Lie\Gamma(K_\infty)),
\end{align*}
so the assumption implies that for every $i \in \ZZ$, the inequality $$\dim_K \gr^i \D_{\dR}(\Lie\Gamma(K_\infty)) \le \dim_K \D_{\C}(\Lie\Gamma(K_\infty)(i))$$ is an equality. Taking in particular $i=0$ implies that $\D_{\dR}^+(\Lie\Gamma(K_\infty)) \to \D_{\C}(\Lie\Gamma(K_\infty))$ is surjective. This shows $\Theta(K_\infty)$ lifts, hence $K_\infty$ is orientable.
\end{proof}

\subsection{Orientations and power series descriptions of $\BdRpluspa$}



 The following theorem shows orientations are in bijection with equivariant power series descriptions of $\BdRpluspa$, answering Question \ref{question:bdr_plus_iso}. It holds even if $K_\infty$ is not assumed to be orientable.

\begin{theorem}
\label{thm:parameterization_of_bdrplus_isos}
The following are in bijection.
\begin{enumerate}
\item  The set of orientations of $K_\infty$.

\item  The set of $\Gamma(K_\infty)$-equivariant sections of the map $\BdRpluspa\to\wh{K}_\infty^\la$.

\item  The set of $\Gamma(K_\infty)$-equivariant isomorphisms  $$\BdRplus(K_\infty)^\pa\cong \wh{\Sym}_{\Kinftyla}(L(K_\infty)^\la)$$
inducing the identity on graded pieces.

\end{enumerate}
\end{theorem}

\begin{proof}
The bijection between (2) and (3) follows from the proof of Theorem \ref{thm:brdplusps_iso}. Let us then construct the bijection between (1) and (2). In the proof of Theorem \ref{thm:brdplusps_iso}, we have given a map $\nabla \mapsto s(\nabla)$ from (1) to (2). It is given as follows: an orientation $\nabla_0$ is mapped to the section $$s(\nabla_0):\wh{K}_\infty^{\la}\cong  \BdRplus(K_\infty)^{\pa,\nabla_0=0}\subset \BdRpluspa.$$ In the inverse direction, given a $\Gamma(K_\infty)$-equivariant section $s_0: \wh{K}_\infty^{\la}\to \BdRpluspa$, we get an induced map $s_{0,*}:\D_{\C}(\Lie\Gamma(K_\infty))\to\DdRplus(\Lie\Gamma(K_\infty))$, and the corresponding orientation is $$\nabla(s_0):=s_{0,*}(\Theta(K_\infty)).$$
It remains to show that the maps $s$ and $\nabla$ are inverses. 

First we show that $s \circ \nabla = \Id.$ Given a section $s_0$, we need to show that $s(\nabla(s_0)) = s_0$. We claim that $s_0,$ which is a priori only a section $\wh{K}_\infty^\la\to\BdRplus(K_\infty)^{\pa}$, actually factors through $\BdRplus(K_\infty)^{\pa,\nabla(s_0)=0}.$ Indeed, write $\nabla(s_0) = \sum b_i \otimes \partial_i$ with the $\partial_i$ forming a basis of $\Lie\Gamma(K_\infty)$ and the $b_i \in \BdRpluspa$. By definition of $\nabla(s_0)$,  each $b_i$ is of the form $s_0(a_i)$ for $a_i \in \Kinftyla$. For $x \in \Kinftyla$ we thus have
\begin{align*}
\nabla(s_0)(s_0(x)) &= \sum b_i\partial_i(s_0(x))\\
&= s_0(\sum a_i \partial_i(x))\\ &= s_0(\Theta(K_\infty)(x)) = 0,
\end{align*}
where the last equality holds by Proposition \ref{prop:sen_operator_kills_la_vecs}. This proves the claim. But now we see that both $s_0$ and $s(\nabla(s_0))$ are inverses of the natural map $\BdRplus(K_\infty)^{\pa,\nabla(s_0)=0}$, so they must be equal.

Conversely, given an orientation $\nabla_0$, we need to show that $\nabla(s(\nabla_0)) = \nabla_0$. Write $\nabla_0 = \sum b_i \otimes \partial_i$ with the $\partial_i$ forming a basis of $\Lie\Gamma(K_\infty)$ and the $b_i \in \BdRpluspa$. Then by Lemma \ref{lem:nabla_kills_its_coefficients}, the elements $b_i$ are killed by $\nabla_0$, hence are in $\BdRplus(K_\infty)^{\pa,\nabla_0=0}$. But this means $\nabla_0 = s(\nabla_0)_{*}(\Theta(K_\infty))$, or $\nabla_0 = \nabla(s(\nabla_0)),$ as required.
\end{proof}

\begin{remark}
    
The tower $K_\infty$ is orientable if and only if $\Theta(K_\infty)$ is mapped to zero under the boundary map
$$\D_\C(\Lie\Gamma(K_\infty))\to \mt{H}^1(G_K,\B_{\dR}^+\otimes_{\QQ_p} \Lie\Gamma(K_\infty)(1)).$$
 Thus the above theorem is saying that the image of $\Theta(K_\infty)$ under this boundary map is the  obstruction for $\BdRpluspa$ to be $\Gamma(K_\infty)$-equivariantly isomorphic to a power series ring over $\Kinftyla$.
\end{remark}

The following lemma was used in proof of Theorem \ref{thm:parameterization_of_bdrplus_isos}. It is the $\bdrplus$-analogue of Theorem  \ref{thm:field_of_def_of_sen_op_is_gamma_finite}.

\begin{lemma}
\label{lem:nabla_kills_its_coefficients} 
Let $\nabla(K_\infty)$ be an orientation of $K_\infty$. Let $\partial_i$ be a $
\QQ_p$-basis of $\Lie\Gamma(K_\infty),$ and write $\nabla(K_\infty) = \sum_i b_i \otimes \partial_i$ with the $b_i \in \BdRpluspa.$ Then for every $i=1,...,d$ we have $\nabla(K_\infty)(b_i) = 0.$ More precisely,
$$\nabla(K_\infty)\in\BdRplus(K_\infty)^{\ker(\ad),\Gamma(K_\infty)\hfin,\nabla(K_\infty)=0}\otimes_{\QQ_p} \Lie\Gamma(K_\infty).$$
\end{lemma}

\begin{proof}
For $g\in \Gamma(K_\infty)$ we have the identity
$$\sum_ig(b_i)\otimes\ad(g)(\partial_i) = \sum_ib_i\otimes\partial_i.$$
Differentiating the identity along $\partial \in \Lie\Gamma(K_\infty)$, we get
$$\sum_i \partial(b_i)\otimes\partial_i + \sum_i b_i \otimes [\partial,\partial_i] = 0.$$
Now take $\partial = \partial_j$, multiply the equation by $b_j$, and sum over $j$. We get
$$\sum_{i,j}b_j\partial_j(b_i)\otimes\partial_i +\sum_{i,j}b_ib_j\otimes[\partial_j,\partial_i] = 0.$$
Since $[\partial_j,\partial_i] = - [\partial_i,\partial_j]$, the second summand is zero. Considering the coefficient of $\partial_i,$ we get
$$\nabla(K_\infty)(b_i) = \sum_jb_j\partial_j(b_i) = 0,$$
which shows the first statement of the theorem. For the stronger statement, argue as in the proof of Theorem \ref{thm:field_of_def_of_sen_op_is_gamma_finite}. 
\end{proof}

\section{Cohomological applications}

In this section we give cohomological applications of our methods. In \S \ref{subsec:sen_modules} we give a version of Sen theory for general $K_\infty.$ Then in \S \ref{subsec:bdrplus_reps} we extend the picture to $\BdRplus$-representations. In both of these subsections the main ingredient is Theorem \ref{thm:higher_la_vanishing_k_infty}. In \S \ref{sec:reg_conn} we establish the deeper link between $\BdRplus$-representations to regular connections when $K_\infty$ is orientable; this uses the full force of Theorem \ref{thm:brdplusps_iso}.

\subsection{Sen modules and cohomology comparisons}
\label{subsec:sen_modules}
 
The main result of this subsection is Theorem \ref{thm:SenBC16}, where we   construct a full analogue of Sen theory---namely, categorical (decompletion)-equivalence, cohomology comparison, etc.~---for a general $p$-adic Lie extension. Many ingredients were known from the work of Sen \cite{sen1980continuous} and Berger--Colmez \cite{berger2016theorie}; our new input here is the vanishing of higher locally analytic vectors which was proved in Theorem \ref{thm:higher_la_vanishing_k_infty}.

 \begin{definition} \label{defsemilinrep}
 Suppose $\mathcal G$ is a topological group that acts continuously on a topological ring $R$. We use $\rep_{\mathcal G}(R)$ to denote the category where an object is a finite free $R$-module $M$ (topologized via the topology on $R$) with a continuous and \emph{semilinear} $\mathcal G$-action in the usual sense that
$$g(rx)=g(r)g(x), \forall g\in \mathcal G, r \in R, x\in M.$$
\end{definition}

The following is the main theorem of this section.

\begin{theorem} \label{thm:SenBC16} 
Let $K_{\infty}/K$ be   an infinitely ramified $p$-adic Lie  extension.
\begin{enumerate}
\item 
There are  equivalences of categories (see Definition \ref{defsemilinrep})
\[ \rep_\gk(\C) \cong \rep_\gammak(\hatkinfty)  \cong \rep_\gammak(\hatkinftyla) \] 
Given $U \in \rep_\gk(\C)$, the corresponding object in $\rep_\gammak(\hatkinfty)$ is $W=U^{\hk}$, and the corresponding object in $\rep_\gammak(\hatkinftyla) $ is $D=W^{\gammak\dla}$. The functors from right to left are base-change functors.

\item For $U \in \rep_\gk(\C)$, we have cohomology comparisons 
\[ \rg(\gk, U) \simeq \rg(\gammak,  W) \simeq
\rg(\gammak,  D) \simeq \rg_\la(\gammak,  D) \]
Here, the first three complexes are continuous group cohomologies, and the last complex denotes   locally analytic cohomology.

\item We have
\[ \rg(\gammak,  D) \otimes_K \kinfty \simeq \rg(\Lie\gammak, D)\]
where the right hand side is Lie algebra cohomology. 
As a consequence,
\[ \rg(\gk, U) \simeq (\rg(\Lie\gammak, D))^\gammak. \]

\item We have
\[ \rg(\gammak,  D) \otimes_K \hatkinftyla \simeq \rg(\Theta(\kinfty), D):=[D \xrightarrow{\Theta(\kinfty)} D].\] 
As a consequence, 
\[ \rg(\gk, U) \simeq  (\rg(\Theta(\kinfty), D))^\gammak. \]
\end{enumerate}
\end{theorem}
\begin{proof}
We first prove items (1) and (2). 
The equivalence $\rep_\gk(\C) \cong \rep_\gammak(\hatkinfty)$ and cohomological comparison  $\rg(\gk, U) \simeq \rg(\gammak,  W)$ follows from almost purity (since $\hatkinfty$ is perfectoid).  The equivalence $\rep_\gammak(\hatkinfty)  \cong \rep_\gammak(\hatkinftyla)$ is proved by Berger--Colmez \cite[Theorem 3.4]{berger2016theorie}. The cohomological comparison
\[ \rg(\gammak,  W) \simeq \rg(\gammak,  D) \simeq \rg_\la(\gammak,  D) \]
follows from the fact that $W$ has no higher locally analytic vectors (Theorem \ref{thm:higher_la_vanishing_k_infty}) and the general comparison theorem \cite[Theorem 1.7]{RJRC22}.

Item (3) follows from the general fact (see~ \cite{Tamme_loc_ana_rep_2015ANT}) that \[\rg(\Lie\gammak, D) \simeq \mathrm{colim}_{K \subset_{\mathrm{fin}} M \subset \kinfty} \rg_\la(\gal(\kinfty/M), D) \]
and Galois descent.
 
For item (4), first note that the $\Theta(\kinfty)$-action on $D$ is   $\hatkinftyla$-linear  since  $\Theta(\kinfty)$ kills $\hatkinftyla$ by Proposition \ref{prop:sen_operator_kills_la_vecs}. 
Let $L_\infty = K_\infty K_\cyc$.   For convenience, use  $D(\kinfty), D(L_\infty)$ and $D(\kcyc)$ to denote the Sen modules for these towers; by Lemma \ref{lem:basis_lav}, we have natural isomorphisms:
\[   D(\kinfty)\otimes_\hatkinftyla \wh{L}_\infty^\la \simeq D(L_\infty) \simeq D(\kcyc) \otimes_\kcyc \wh{L}_\infty^\la.\]
Recall Lemma \ref{lemma:sen_operator_compatibility} proved compatibilities between $\Theta(L_\infty)$, $\Theta(\kinfty)$ and $\Theta(\kcyc)$. 
We thus have
\begin{align*} \label{eq:compa_sen}
  \rg(\Theta(\kinfty), D(\kinfty)) \otimes_{\hatkinftyla} \wh{L}_\infty^\la &\simeq
 \rg(\Theta(L_\infty), D(L_\infty)) \\&\simeq  \rg(\Theta(\kcyc), D(\kcyc)) \otimes_{\kcyc} \wh{L}_\infty^\la. 
\end{align*} 
Indeed, these complexes are   computing the same endomorphism using different bases. 
 Recall the classical (cyclotomic) fact \[\rg(\Theta(\kcyc), D(\kcyc)) \simeq \rg(\gk, U)\otimes_K \kcyc.\]
We can thus take $\gal(L_\infty/\kinfty)$-invariants on the previous quasi-isomorphism to conclude. 
\end{proof}

\subsection{D\'evissage of $\bdrplus$-representations}
\label{subsec:bdrplus_reps}

The goal of this subsection is to prove     Theorem \ref{thm:dR_cat}, which generalizes items (1)-(3) of Theorem \ref{thm:SenBC16} to the $\bdrplus$-setting, for   \emph{any} $p$-adic Lie extension $\kinfty/K$.
As we see,  a generalization of Theorem \ref{thm:SenBC16}(4)  (concerning cohomology of the Sen operator) is missing here: this will have to make full use of Theorem \ref{thm:brdplusps_iso}, under the \emph{orientability} assumption; cf.~ Theorem \ref{thm:dR_cat_coho}.

 The content of the following theorem is hinted above \cite[Theorem 3.4]{berger2016theorie}.
 
\begin{theorem} \label{thm:dR_cat} 
Let $K_{\infty}/K$ be   an infinitely ramified $p$-adic Lie  extension.
\begin{enumerate}
\item 
There are  equivalences of categories
\[ \rep_\gk(\mathbf B_{\mathrm{dR}}^+) \simeq \rep_\gammak(\mathbf B_{\mathrm{dR}}^+(K_\infty))  \simeq \rep_\gammak(\BdRpluspa). \] 
Given $U \in \rep_\gk(\mathbf B_{\mathrm{dR}}^+)$, the corresponding object in $\rep_\gammak(\mathbf B_{\mathrm{dR}}^+(K_\infty))$ is $W=U^{\hk}$, and the corresponding object in $\rep_\gammak(\BdRpluspa) $ is $D=W^{\gammak\dpa}$. The functors from right to left are base-change functors.

\item For $U \in \rep_\gk(\mathbf B_{\mathrm{dR}}^+)$, we have cohomology comparisons 
\[ \rg(\gk, U) \simeq \rg(\gammak,  W) \simeq
\rg(\gammak,  D) \simeq  \varprojlim_n \rg_\la(\gammak,  D/t_{\kinfty}^n D) \]
Here, the first three complexes are continuous group cohomologies, and the last complex denotes inverse limit of  locally analytic cohomology.

\item We have
\[ \rg(\gammak,  D) \otimes_K \kinfty \simeq \rg(\Lie\gammak, D)\]
where the right hand side is Lie algebra cohomology. 
As a consequence
\[ \rg(\gk, U) \simeq (\rg(\Lie\gammak, D))^\gammak \] 
 \end{enumerate}
\end{theorem} 

\begin{proof} 
For each $m\in\mathbb{Z}_{\geq1}$, let $\B_{\dR,m}^{+}=\B_{\dR}^{+}/I_\theta^{m}$; by a standard d\'evissage argument, it suffices to prove that we have   equivalences of categories: \begin{equation} \label{eq:equiv_level_m}
\rep_\gk(\bdrplusm) \simeq \rep_\gammak(\bdrplusm(K_\infty))  \simeq \rep_\gammak(\bdrplusmkinftyla). 
\end{equation} 
via the analogous functors
\[ U_m \mapsto W_m:=(U_m)^\hk, \quad W_m \mapsto D_m:=(W_m)^{\gammak\dla}.\]
Indeed, assuming these equivalences, the cohomological comparisons (for $U_m, W_m, D_m$) follow by obvious d\'evissage argument to the $m=1$ case, which is already established in Theorem \ref{thm:SenBC16}. 

We first prove the first   equivalence  in \eqref{eq:equiv_level_m} via induction on $m$ with the $m=1$ case known in  Theorem \ref{thm:SenBC16}.
 It is clear that we only need to prove that   $W_m =(U_m)^\hk$ is finite free over $\bdrplusm(\kinfty)$ and the natural map 
\begin{equation} \label{wmum}
W_m \otimes_{\bdrplusm(\kinfty)} \bdrplusm \to U_m
\end{equation}
  is an isomorphism.
For brevity, for any $k \leq m$, let $U_k:=U_m\otimes_{\bdrplusm } \bdrplusm/I_\theta^k$ and similarly for $W_k, D_k$.
 Theorem \ref{thm:SenBC16} implies that  higher (i.e., degree $\geq 1$) $\hk$-cohomology of $U_1$ vanishes; thus so is that of $U_k$ for any $k$. Thus the short exact sequence
 \[0 \to U_{m-1} \xrightarrow{\times t} U_m \to U_1 \to 0\]
induces a short exact sequence
\[0 \to W_{m-1} \to W_m \to W_1 \to 0\]
This (via short five lemma, and induction hypothesis) implies that \eqref{wmum} is an isomorphism; this forces $W_m$ to be finite free since $\bdrplus(\kinfty)$ is a DVR.

We now prove the second   equivalence  in \eqref{eq:equiv_level_m} via induction on $m$ with $m=1$ case known in  Theorem \ref{thm:SenBC16}. The strategy is quite similar to that in the preceding paragraph. It suffices to prove $D_m$ is finite free over $\bdrplusmkinftyla$ and natural map
\begin{equation} \label{dmwm}
 D_m\otimes_{\bdrplusmkinftyla} \bdrplus(\kinfty) \to W_m
\end{equation}
is an isomorphism.
Corollary \ref{cor:higher_la_vanishing_bdr} implies   $W_k$ has no higher locally analytic vectors. Thus  we have a short exact   sequence
\[ 0 \to D_{m-1} \to D_m \to D_1 \to 0.\]
  Thus \eqref{dmwm} is an isomorphism by induction and short five lemma. Again, this forces $ D_m$ to be finite free since $\BdRpluspa$ is a DVR by Theorem \ref{thm:dvr}. 
\end{proof}

\subsection{Orientations and regular connections} \label{sec:reg_conn}

 In this subsection, when $\kinfty$ is orientable, we relate $\bdrplus$-representations with  regular connections. 

 \begin{definition}
      Let $E$ be a characteristic zero field.  Let $z$ be a variable. Equip $E\llbracket z\rrbracket$ with the differentiation $z\frac{d}{dz}$. A regular  connection over $E\llbracket z \rrbracket$ is a  finite free $E\llbracket z \rrbracket$-module $M$ equipped with an $E$-linear morphism $\nabla: M \to M$ satisfying the Leibniz rule with respect to $z\frac{d}{dz}$. Given a regular connection, denote \[\rg(\nabla, M): = [M \xrightarrow{\nabla} M].\]
 \end{definition}


\begin{theorem} \label{thm:dR_cat_coho}  
Use the same notation as in Theorem \ref{thm:dR_cat}, and suppose furthermore $\kinfty$ is orientable with a chosen orientation $\nabla(\kinfty)$. Let  $U \in \rep_\gk(\mathbf B_{\mathrm{dR}}^+)$ with corresponding $D \in \rep_\gammak(\BdRpluspa)$. 
Then the pair $(D, \nabla(\kinfty))$ is a \emph{regular connection} over the ring $\BdRpluspa$. In addition, we have
\begin{equation}\label{eqn:coho_reg_conn}
    \rg(\gammak,  D) \otimes_K \bdrplus(\kinfty)^{\pa, \nabla=0} \simeq \rg(\nabla(\kinfty), D)
\end{equation}  
As a consequence 
\begin{equation} \label{eqn:coho_reg_gal_inv}
     \rg(\gk, U) \simeq  (\rg(\nabla(\kinfty), D))^\gammak. 
\end{equation}
\end{theorem}
\begin{proof}
 We know $\nabla(t_\kinfty)=t_\kinfty$ from Theorem \ref{thm:brdplusps_iso}, thus clearly $(D, \nabla)$ is a regular connection. We focus on proving \eqref{eqn:coho_reg_conn} as \eqref{eqn:coho_reg_gal_inv} is trivial consequence.

The idea of the proof is quite similar to that of Theorem \ref{thm:SenBC16}(4) (i.e., using $\kcyc$ and its composite with $\kinfty$), except we need to upgrade the ``linear endomorphism" base change comparison   $$ \rg(\Theta(\kinfty), D(\kinfty)) \otimes_{\hatkinftyla} \wh{L}_\infty^\la \simeq  \rg(\Theta(\kcyc), D(\kcyc)) \otimes_{\kcyc} \wh{L}_\infty^\la. $$ to a base change comparison of \emph{regular connections}.

 Let $L_\infty=\kinfty\kcyc$.   Lemma \ref{lem:orientations_preserved_under_central_extensions} implies $\nabla(\kinfty)$ induces a unique  compatible orientation $\nabla(L_\infty)$, which then necessarily specializes to the unique cyclotomic orientation $\nabla(\kcyc)$ by Lemma \ref{lem:nabla_action_compatibility}.
Use Theorem \ref{thm:brdplusps_iso}(1) to identify $\bdrplus(\kinfty)^{\pa, \nabla=0}$ with $\hatkinftyla$; similarly for the towers $\kcyc$ and $L_\infty$. The compatibilities among all these orientations imply that we have a  commutative diagram
\[
\begin{tikzcd}
{\hatkinftyla\llbracket t_{K_\infty}\rrbracket} \arrow[d, hook] \arrow[r, "\simeq"] & \bdrplus(\kinfty)^\pa \arrow[d, hook] \\
{\wh{L}_\infty^\la\llbracket t\rrbracket} \arrow[r, "\simeq"]                   & \bdrplus(L_\infty)^\pa                \\
{\kcyc\llbracket t\rrbracket} \arrow[u, hook] \arrow[r, "\simeq"]               & \bdrplus(\kcyc)^\pa \arrow[u, hook]  
\end{tikzcd}
\]
where all vertical arrows are inclusion maps.
Use  $D(\kinfty), D(L_\infty)$ and $D(\kcyc)$ to denote the regular connection for these towers, then Lemma \ref{lem:basis_lav} implies:
\[   D(\kinfty)\otimes_{\hatkinftyla\llbracket t_\kinfty\rrbracket} \wh{L}_\infty^\la\llbracket t\rrbracket \simeq D(L_\infty) \simeq D(\kcyc) \otimes_{\kcyc\llbracket t\rrbracket} \wh{L}_\infty^\la\llbracket t\rrbracket.\]
Thus base change property for cohomology of regular connections (e.g., see \cite[Proposition 8.7]{GWocdr} which is elementary) imply
\begin{align*}
\label{eq:compa_dR}
  \rg(\nabla(\kinfty), D(\kinfty)) \otimes_{\hatkinftyla} \wh{L}_\infty^\la &\simeq
 \rg(\nabla(L_\infty), D(L_\infty))\\ &\simeq  \rg(\nabla(\kcyc), D(\kcyc)) \otimes_{\kcyc} \wh{L}_\infty^\la. 
\end{align*} 
Using  the classical  cyclotomic  fact (see~e.g.~ \cite[Theorem 10.12]{GMWdR}) \[\rg(\nabla(\kcyc), D(\kcyc)) \simeq \rg(\gk, U)\otimes_K \kcyc,\]
we can thus take $\gal(L_\infty/\kinfty)$-invariants on the previous quasi-isomorphism    to conclude. 
\end{proof}

 \bibliographystyle{alpha}
\bibliography{main}
\end{document}